\documentclass{article}
\usepackage[T1]{fontenc}
\usepackage[latin9]{inputenc}
\usepackage{amsthm}
\usepackage{amsmath}
\usepackage{graphicx}
\usepackage{amssymb}
\usepackage[all]{xy}
\usepackage[unicode=true, pdfusetitle,
 bookmarks=true,bookmarksnumbered=false,bookmarksopen=false,
 breaklinks=false,pdfborder={0 0 1},backref=false,colorlinks=false]
 {hyperref}

\providecommand{\tabularnewline}{\\}

\theoremstyle{plain}
  \theoremstyle{plain}
  \newtheorem*{thm*}{Theorem}
  \theoremstyle{plain}
  \newtheorem{thm}{Theorem}[section]
  \theoremstyle{remark}
  \newtheorem*{acknowledgement*}{Acknowledgement}
  \theoremstyle{definition}
  \newtheorem{defn}[thm]{Definition}
  \theoremstyle{remark}
  \newtheorem{rem}[thm]{Remark}
 \theoremstyle{definition}
  \newtheorem{example}[thm]{Example}
  \theoremstyle{plain}
  \newtheorem{prop}[thm]{Proposition}
  \theoremstyle{plain}
  \newtheorem{lem}[thm]{Lemma}
  \theoremstyle{plain}
  \newtheorem{cor}[thm]{Corollary}
  \theoremstyle{plain}
  \newtheorem{conjecture}[thm]{Conjecture}
  \theoremstyle{plain}
  \newtheorem{question}[thm]{Question}

\title{A bordered Chekanov-Eliashberg algebra}
\author{Steven Sivek}

\begin{document}

\maketitle
\begin{abstract}
Given a front projection of a Legendrian knot $K$ in $\mathbb{R}^{3}$
which has been cut into several pieces along vertical lines, we assign
a differential graded algebra to each piece and prove a van Kampen
theorem describing the Chekanov-Eliashberg invariant of $K$ as a
pushout of these algebras. We then use this theorem to construct maps
between the invariants of Legendrian knots related by certain tangle
replacements, and to describe the linearized contact homology of Legendrian
Whitehead doubles. Other consequences include a Mayer-Vietoris sequence
for linearized contact homology and a van Kampen theorem for the characteristic
algebra of a Legendrian knot.
\end{abstract}

\section{Introduction}

\subsection{The Chekanov-Eliashberg invariant}

Let $(\mathbb{R}^{3},\xi)$ denote the standard contact structure
$\xi=\ker(dz-ydx)$ on $\mathbb{R}^{3}$. A knot $K\subset\mathbb{R}^{3}$
is said to be Legendrian if $T_{x}K\subset\xi_{x}$ at every point
of $K$, and two knots $K_{0}$ and $K_{1}$ are Legendrian isotopic
if they are connected by a family $K_{t}$ of Legendrian knots.

Chekanov \cite{Chekanov:2002p539} defined for each Legendrian knot
$K\subset(\mathbb{R}^{3},\xi)$ an associative unital differential
graded algebra (DGA), here denoted $Ch(K)$, whose stable tame isomorphism
type is an invariant of $K$ up to Legendrian isotopy. Given a Lagrangian
projection of $K$, i.e. a projection of $K$ onto the $xy$-plane,
the algebra is generated freely over $\mathbb{F}=\mathbb{Z}/2\mathbb{Z}$
by the crossings of $K$, which correspond to Reeb chords in $(\mathbb{R}^{3},\xi)$,
and graded by $\mathbb{Z}/2r(K)\mathbb{Z}$, where $r(K)$ is the
rotation number of $K$. (Etnyre, Ng, and Sabloff \cite{Etnyre:2002p544}
later extended the base ring to $\mathbb{Z}[t,t^{-1}]$ and the grading
to a full $\mathbb{Z}$ grading.) The differential counts certain
immersed disks in the knot diagram, and although it was motivated
by contact homology \cite{Eliashberg:1998p718} its computation is
entirely combinatorial. Thus $Ch(K)$ could be used to distinguish
between two Legendrian representatives of the $5_{2}$ knot even though their
classical invariants $tb$ and $r$ are the same.

Legendrian knots are often specified by front projections, which are
projections onto the $xz$-plane. A knot can be uniquely recovered
from its front projection since the $y$-coordinate at any point is
the slope $\frac{dz}{dx}$; in particular the projection has no vertical
tangent lines, so at each critical point of $x$ there is a cusp.
At any crossing the segment with smaller slope passes over the one
with larger slope. Ng \cite{Ng:2003p540} gave a construction of $Ch(K)$
for front projections, and showed that given a so-called {}``simple''
front the DGA is very easy to describe.

Meanwhile, on the way to constructing bordered Heegaard Floer homology
\cite{Lipshitz:2008p649} as an invariant of 3-manifolds with marked
boundary, Lipshitz, Ozsv{\'a}th, and Thurston constructed a simplified
model of knot Floer homology for bordered grid diagrams \cite{Lipshitz:2008p652}.
By cutting a grid diagram along a vertical line, they associate differential
modules $CPA^{-}(\mathcal{H}^{A})$ and $CPD^{-}(\mathcal{H}^{D})$
over some algebra $\mathcal{A}$ to the two halves $\mathcal{H}^{A}$
and $\mathcal{H}^{D}$ of the diagram $\mathcal{H}$ so that their
tensor product is the {}``planar Floer homology'' $CP^{-}(\mathcal{H})$.
Since the differential on $CP^{-}$ counts certain rectangles in the
grid diagram, the algebra $\mathcal{A}$ is constructed to remember
when these rectangles cross the dividing line, and so the pairing
theorem \[
CPA^{-}(\mathcal{H}^{A})\otimes_{\mathcal{A}}CPD^{-}(\mathcal{H}^{D})\cong CP^{-}(\mathcal{H})\]
is a straightforward consequence of the construction. However, the
chain complex $CP^{-}(\mathcal{H})$ is not an invariant of the underlying
knot, and a similar decomposition for the knot Floer homology complex
$CFK^{-}$ seems to be significantly harder.

Our goal in this paper is to present a similar decomposition theorem
for the Chekanov-Eliashberg DGA associated to a front diagram. By
dividing a simple front into left and right halves $K^{A}$ and $K^{D}$
which intersect the dividing line in $n$ points we will construct
two DGAs, $A(K^{A})$ and $D(K^{D})$.  These DGAs admit morphisms into them from
another DGA denoted $I_{n}$, where a DGA morphism is an algebra homomorphism
which preserves gradings and satisfies $\partial\circ f = f\circ\partial$.
We then prove the following analogue of van Kampen's theorem:
\begin{thm*}
The commutative diagram \[
\xymatrix{I_{n}\ar[r]\ar[d] & D(K^{D})\ar[d]\\
A(K^{A})\ar[r] & Ch(K)}
\]
is a pushout square in the category of DGAs.
\end{thm*}
This theorem adds to the {}``algebraic topology'' picture of the
Chekanov-Eliashberg algebra which originated with Sabloff's Poincar\'{e} 
duality theorem \cite{Sabloff:2006p526} and also includes cup products,
Massey products, and $A_{\infty}$ product structures \cite{Civan:2009p1129};
these previous results all apply cohomological ideas to linearizations
of the DGA, whereas the van Kampen theorem suggests that the DGA should
be thought of as a {}``fundamental group'' of a Legendrian knot. 

After developing the van Kampen theorem and generalizing it to further
divisions of Legendrian fronts, we add to the cohomological picture
by constructing a related Mayer-Vietoris sequence in linearized contact
homology. We will then use these ideas to construct morphisms between
the DGAs of some Legendrian knots related by tangle replacements,
and in particular apply these techniques to understand augmentations
of Legendrian Whitehead doubles. Finally, we make some similar observations
about the closely related characteristic algebra.

\subsection{The algebra of a simple Legendrian front}

This section will review the construction of the Chekanov-Eliashberg
DGA for a Legendrian front as in \cite{Ng:2001p667,Ng:2003p540}.
Although it can be constructed for any front,
we will restrict our attention to simple fronts, where the DGA is particularly
easy to describe.
Throughout
this paper all DGAs will be assumed to be {\em semi-free}
\cite{Chekanov:2002p539}, i.e. freely generated over $\mathbb{F}$
by a specified set of generators.
\begin{defn}
\label{def:simple-front}A Legendrian front is \emph{simple} if it can be
changed by a planar isotopy so that all of its right cusps have the same
$x$-coordinate.\end{defn}
\begin{rem}
We will also describe a piece of a front cut out by two vertical lines
as simple if no right cusp lies in a compact region bounded by the
front and the vertical lines; this will ensure that these pieces form
a simple front when glued together.
\end{rem}
\begin{figure}
\begin{centering}
\includegraphics{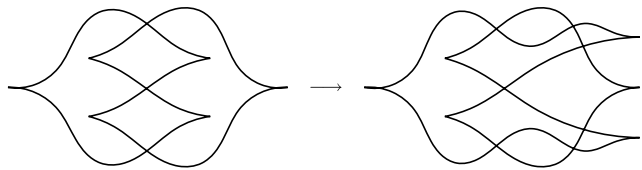}
\par\end{centering}

\caption{A front diagram of a Legendrian trefoil is made simple by pulling
the two interior right cusps rightward and using Legendrian Reidemeister
moves.\label{Flo:legendrian-trefoil}}

\end{figure}
Two fronts represent the same Legendrian knot if and only if they
are related by a sequence of Legendrian Reidemeister moves \cite{Swiatkowski:1992p1071}:

\begin{center}
\includegraphics{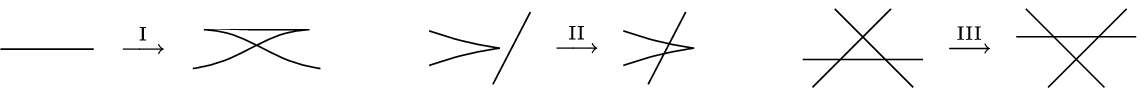}
\par\end{center}

\noindent Therefore every Legendrian knot admits a simple representative
by taking an arbitrary front and using type II Reidemeister moves
to pull each right cusp outside of any compact region, as in Figure
\ref{Flo:legendrian-trefoil}, though this will increase the number
of crossings.
\begin{defn}
The \emph{vertices} of a simple Legendrian front are its crossings
and right cusps.
\end{defn}
The simple front on the right side of Figure \ref{Flo:legendrian-trefoil}
has ten vertices: there are seven crossings and three right cusps.
\begin{defn}
An \emph{admissible disk} for a vertex $v$ of a simple front $K$
is a disk $D^{2}\subset\mathbb{R}^{2}$ with $\partial D\subset K$
satisfying the following properties:
\begin{enumerate}
\item $D$ is smoothly embedded except possibly at vertices and left cusps;
\item The vertex $v$ is the unique rightmost point of $D$;
\item $D$ has a unique leftmost point at a left cusp of $K$;
\item At any \emph{corner} of $D$, i.e. a crossing $c\not=v$ where $D$
is singular, a small neighborhood $U$ of $c$ is divided into four
regions by $U\cap K$; we require that $U\cap D$ be contained in
exactly one of these regions. 
\end{enumerate}
Let $\mathrm{Disk}(K;v)$ denote the set of admissible disks for the
vertex $v$.
\end{defn}
\begin{figure}
\begin{centering}
\includegraphics{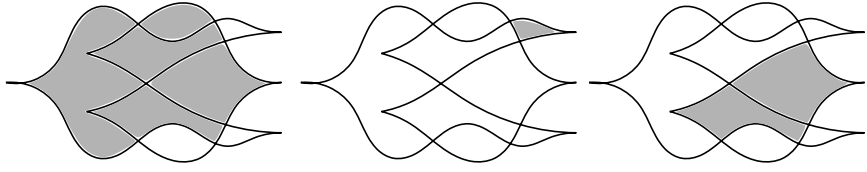}
\par\end{centering}

\caption{Disks embedded in the simple front diagram of Figure \ref{Flo:legendrian-trefoil}.
The first two are not admissible -- one occupies three quadrants around
the top middle crossing, and one does not have its leftmost point
at a left cusp -- but the last one is admissible.\label{fig:admissible-disks}}

\end{figure}

\begin{defn}
The Chekanov-Eliashberg algebra of a simple front $K$, denoted $Ch(K)$,
is the DGA generated freely over $\mathbb{F}=\mathbb{Z}/2\mathbb{Z}$
by the vertices of $K$. Its differential is given by \[
\partial c=\begin{cases}
\sum_{D\in\mathrm{Disk}(K;c)}\partial D, & c\mathrm{\ a\ crossing}\\
1+\sum_{D\in\mathrm{Disk}(K;c)}\partial D, & c\mathrm{\ a\ right\ cusp},\end{cases}\]
where $\partial D$ denotes the product of the corners of $D$ as
seen in counterclockwise order from $c$.

If $K$ has rotation number $r(K)$, we can assign a \emph{Maslov
potential} $\mu(s)\in\Gamma=\mathbb{Z}/2r(K)\mathbb{Z}$ to each strand
$s$ of $K$ so that at any left or right cusp, the top strand $s_{1}$
and bottom strand $s_{2}$ satisfy $\mu(s_{1})-\mu(s_{2})=1$. Then
$Ch(K)$ admits a $\Gamma$-grading in which each right cusp has grading
$|c|=1$, and at each crossing $c$ with top strand $s_{1}$ crossing
over the bottom strand $s_{2}$ we define the grading to be $|c|=\mu(s_{1})-\mu(s_{2})$.  (Recall that in a front projection, the strand with smaller slope
always crosses over the strand with larger slope.)
\end{defn}
\begin{rem}
The grading is well-defined in $\mathbb{Z}/2r(K)\mathbb{Z}$ for
knots but ambiguous for links, since we may change the Maslov potential on
every strand of a single component $K$ by some constant $c$ and thus change the
gradings at every vertex where exactly one strand belongs to $K$ by $\pm c$.
In practice we will always work with an explicit choice of grading.
\end{rem}
\begin{figure}
\begin{centering}
\includegraphics{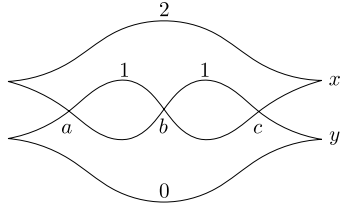}
\par\end{centering}

\caption{A simple front for another Legendrian trefoil, with vertices and Maslov
potentials labeled.\label{fig:simple-trefoil}}

\end{figure}

\begin{example}
If $K$ is the simple front of Figure \ref{fig:simple-trefoil}, then
$Ch(K)$ is generated freely by $a,b,c,x,y$ satisfying \begin{eqnarray*}
\partial x & = & 1+abc+a+c\\
\partial y & = & 1+cba+c+a\\
\partial a=\partial b=\partial c & = & 0.\end{eqnarray*}
The Maslov potentials indicated in Figure \ref{fig:simple-trefoil}
give $Ch(K)$ a $\mathbb{Z}$-grading with $|x|=|y|=1$ and $|a|=|b|=|c|=0$.\end{example}
\begin{defn}
A \emph{tame isomorphism} $\mathcal{A}\to\mathcal{A}'$ of DGAs with
free generators $g_{1},\dots,g_{n}$ and $g'_{1},\dots,g'_{n}$ is
an automorphism of $\mathcal{A}$ of the form
\[
g_{i}\mapsto g_{i}+\varphi(g_{1},\dots,g_{i-1},g_{i+1},\dots g_{n})\]
which fixes all other $g_{j}$, followed by the isomorphism
$g_{i}\mapsto g'_{i}$ for all $i$. A \emph{stabilization} of the DGA
$\mathcal{A}$ preserves all generators
and differentials and adds two new generators $a,b$ in gradings
$k+1$ and $k$ for some $k$, satisfying $\partial a=b$ and $\partial b=0$. Two
DGAs are said to be \emph{stable tame isomorphic} if they are related
by a sequence of tame isomorphisms, stabilizations, and destabilizations.\end{defn}
\begin{thm}[\cite{Chekanov:2002p539,Ng:2003p540}]
 The differential $\partial$ on $Ch(K)$ satisfies $\partial^{2}=0$
and lowers degree by 1, and the stable tame isomorphism type of $Ch(K)$
is an invariant of $K$ up to Legendrian isotopy.
\end{thm}
Finally, we will outline the proof from \cite{Ng:2001p667} that $\partial^{2}=0$,
since we will use slight variations of this argument repeatedly in
the following sections. For any vertex $c$ of $K$, a monomial of
$\partial c$ is the product $c_{1}c_{2}\dots c_{k}$ of corners along
the boundary of a disk $D\in\mathrm{Disk}(K;c)$, and the corresponding
terms of $\partial^{2}c$ involve replacing some $c_{i}$ in that
product with $\partial c_{i}$. Since $\partial c_{i}$ is the sum
of terms $\partial D'$ over disks $D'\in\mathrm{Disk}(K;c_{i})$,
the monomials of $\partial^{2}c$ are products of corners of certain
regions $R=D\cup D'$. In $R$, the disks $D$ and $D'$ intersect
only along a segment of a strand through $c_{i}$; at the other endpoint
$c'$ of $D\cap D'$, the region $R$ occupies three of four quadrants;
and $R$ has two left cusps, one coming from each of $D$ and $D'$.

\begin{figure}
\begin{centering}
\includegraphics{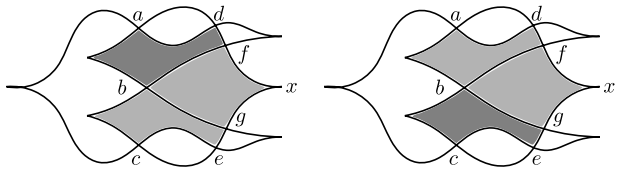}
\par\end{centering}

\caption{Two ways to split a region appearing in the proof that $\partial^{2}=0$.\label{fig:ch-diff-squared}}

\end{figure}
Figure \ref{fig:ch-diff-squared} shows an example of such a region
$R$ appearing in the computation of $\partial^{2}x$ for the simple
front of Figure \ref{Flo:legendrian-trefoil}. On the left, the lighter
disk gives the monomial $fce$ of $\partial x$, and differentiating
this at $f$ gives us a term $(dab)ce$ of $\partial^{2}x$ via the
darker disk. On the right, however, the lighter disk gives the monomial
$dag$ of $\partial x$, and differentiating at $g$ contributes a
term $da(bce)$ from the darker disk. Thus the term $dabce$ appears
twice in $\partial^{2}x$, and since $Ch(K)$ is defined over $\mathbb{Z}/2\mathbb{Z}$
these terms sum to zero.

This argument works in general: following either of the two strands
through the point $c'$ ($b$ in Figure \ref{fig:ch-diff-squared})
until it intersects $\partial R$ again (at $f$ or $g$ in Figure
\ref{fig:ch-diff-squared}) gives us exactly two ways to split $R$
into a union of disks $D\cup D'$ which contribute the same monomial
to $\partial^{2}c$. Since the terms of $\partial^{2}c$ cancel in
pairs, we must have $\partial^{2}c=0$.

\section{The bordered Chekanov-Eliashberg algebra}

\subsection{The algebra of a finite set of points}

Let $n$ be a nonnegative integer, and suppose we have a vertical dividing
line which intersects a front in $n$ points. (Note that $n$ will
always be even in practice, but we do not need this assumption for
now.) Furthermore, suppose we have a potential function $\mu:\{1,2,\dots,n\}\to\Gamma$,
where $\Gamma$ is a cyclic group such as $\mathbb{Z}$ or $\mathbb{Z}/2\mathbb{Z}$.
\begin{defn}
The algebra $I_{n}^{\mu}$ is the DGA generated freely over $\mathbb{F}$
by elements $\{\rho_{ij}\mid1\leq i<j\leq n\}$ with grading $|\rho_{ij}|=\mu(i)-\mu(j)-1$.
It has a differential defined on these generators as \[
\partial\rho_{ij}=\sum_{i<k<j}\rho_{ik}\rho_{kj}\]
and extended to all of $I_{n}^{\mu}$ by the Leibniz rule.
\end{defn}
Although the grading depends on $\mu$, we will in general omit it
from the notation and simply write $I_{n}$.
\begin{prop}
The differential $\partial$ lowers the grading by $-1$ and satisfies
$\partial^{2}=0$.\end{prop}
\begin{proof}
Both assertions follow by a straightforward calculation.
\end{proof}
\begin{rem}
This algebra appears in \cite{Mishachev:2003p676} as the {}``interval
algebra'' $I_{n}(n)$, where a closely related construction determines
the DGA of the $n$-copy of a topological unknot or of a negative
torus knot.
\end{rem}
The purpose of this algebra is to remember where disks that might
be counted by a differential cross the dividing line: if the boundary
of a disk starts on the right side of the line and crosses it at the
$i$th and $j$th points, we will use the element $\rho_{ij}$ as
a placeholder for the contribution to the boundary of the disk from
the left side of the dividing line.

\subsection{The type A algebra\label{sec:type-A-algebra}}

Let $K^{A}$ be the left half of a simple Legendrian front diagram
divided along some fixed vertical line, and suppose we have a Maslov
potential $\mu$ assigning an element of the cyclic group $\Gamma$
to each strand of $K^{A}$.

\begin{figure}
\begin{centering}
\includegraphics{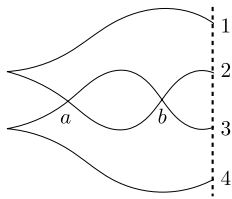}
\par\end{centering}

\caption{A half-diagram $K^{A}$ constructed from the trefoil of Figure \ref{fig:simple-trefoil}.}
\label{Flo:a-trefoil}

\end{figure}

\begin{defn}
The \emph{type A} algebra $A(K^{A})$ is the DGA generated freely
over $\mathbb{F}$ by the vertices of $K^{A}$. Each cusp has grading
$|c|=1$, and if a crossing $c$ has top strand $s_{1}$ and bottom
strand $s_{2}$, then its grading is $|c|=\mu(s_{1})-\mu(s_{2})$.

We define a differential $\partial$ on $A(K^{A})$ exactly as in
the original algebra $Ch(K)$: $\partial c=\sum\partial D$ if $c$
is a crossing and $\partial c=1+\sum\partial D$ if $c$ is a cusp,
where $D$ ranges over all disks in $\mathrm{Disk}(K^{A};c)$.
\end{defn}
The differential is clearly well-defined, since for any vertex $c$
of $K^{A}$ each term $\partial D$ in $\partial c$ is a monomial
consisting of vertices to the left of $c$ and these vertices are
all in $K^{A}$. Furthermore, $\partial^{2}=0$ on $A(K^{A})$ since
the differential on $Ch(K)$ also satisfies $\partial^{2}=0$ and
$A(K^{A})$ is a subalgebra of $Ch(K)$.

Although $A(K^{A})$ seems fairly uninteresting on its own, if the
dividing line intersects it in $n$ points, numbered in order from
$x_{1}$ at the top to $x_{n}$ at the bottom, then $A(K^{A})$ admits
a useful map from $I_{n}^{\mu}$. By giving $I_{n}$ the potential
$\mu$, we mean that the potential at $x_{i}$ should equal the potential
of the corresponding strand of $K^{A}$.
\begin{defn}
\label{def:half-disks-a}Let $\mathrm{Half}_{A}(K^{A};i,j)$ be the
set of admissible embedded left half-disks in $K^{A}$.  These are defined
identically to admissible disks, but instead of having a unique rightmost
vertex we require the rightmost part of the boundary to be
the segment of the dividing line from $x_{i}$ to
$x_{j}$.  For such a half-disk $H$, we define the monomial $\partial H$
to be the product of its corners in $K^{A}$, traversed in counterclockwise
order from $x_{i}$ to $x_{j}$. 
\end{defn}
We can now define an algebra homomorphism $w:I_{n}\to A(K^{A})$ by
the formula \[
w(\rho_{ij})=\sum_{H\in\mathrm{Half}_{A}(K^{A};i,j)}\partial H.\]
For example, in Figure \ref{Flo:a-trefoil} the algebra $A(K^{A})$
is generated freely by $a$ and $b$ with $\partial a=\partial b=0$,
and we can compute the values of $w$ as follows:

\begin{center}
\begin{tabular}{lll}
$w(\rho_{12})=ab+1$ & $w(\rho_{14})=0$ & $w(\rho_{24})=a$\tabularnewline
$w(\rho_{13})=a$ & $w(\rho_{23})=0$ & $w(\rho_{34})=ba+1$\tabularnewline
\end{tabular}
\par\end{center}
\begin{lem}
\label{lem:w-degree-zero}The map $w$ preserves gradings, i.e. $|\rho_{ij}|=|w(\rho_{ij})|$.\end{lem}
\begin{proof}
Any half-disk $H\in\mathrm{Half}_{A}(K^{A};i,j)$ has leftmost point
at a left cusp $y$. As we follow the boundary of $H$ from $x_{i}$
to $y$, we change strands in $K^{A}$ at corners $c_{1},c_{2},\dots,c_{k}$,
and then while following from $y$ to $x_{j}$ we change strands at
corners $c_{1}',c_{2}',\dots,c_{l}'$; by definition $\partial H=c_{1}\dots c_{k}c_{1}'\dots c_{l}'$.
Now the difference in potential between $x_{i}$ and the top strand
$s_{1}$ at $y$ is $|c_{1}|+\dots+|c_{k}|$, and the difference between
the bottom strand $s_{2}$ at $y$ and $x_{j}$ is $|c_{1}'|+\dots+|c_{l}'|$,
hence \[
(\mu(x_{i})-\mu(s_{1}))+(\mu(s_{2})-\mu(x_{j}))=\sum_{l}|c_{l}|+\sum_{l'}|c_{l}'|.\]
But the left hand side is $\mu(x_{i})-\mu(x_{j})-1=|\rho_{ij}|$ since
$\mu(s_{1})=\mu(s_{2})+1$, and the right hand side is $|\partial H|$,
so we are done.\end{proof}
\begin{prop}
\label{pro:w-chain-map}The map $w$ is a chain map.\end{prop}
\begin{proof}
We need to check that $w(\partial\rho_{ij})=\partial w(\rho_{ij})$
for each $i,j$. Letting $w_{ij}=w(\rho_{ij})$ for convenience, this
is the assertion that \[
\partial w_{ij}=\sum_{i<k<j}w_{ik}w_{kj}.\]
The element $w_{ij}\in A(K^{A})$ is a sum of monomials corresponding
to the boundaries of half-disks $H$, so the monomials in $\partial w_{ij}$
are precisely those obtained by taking such an $H$ and gluing it
to full disks which start at a corner $c$ of $\partial H$. The boundary
of the resulting region $R$ goes from $x_{i}$ to a left cusp, back
to a vertex $c'$ where $R$ occupies three of the four adjacent quadrants,
to another left cusp, and then right to $x_{j}$ and back to $x_{i}$
along the dividing line; the associated monomial in $\partial w_{ij}$
is the product of all corners of the disk except for $c'$.

\begin{figure}
\begin{centering}
\includegraphics{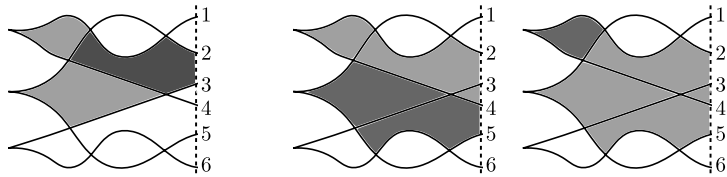}
\par\end{centering}

\caption{The region on the left can be broken into disks representing two monomials
of $\partial w(\rho_{23})$, by merging the dark region with either
light one to get the corresponding monomial of $w(\rho_{23})$; at
center and right the region is broken into pieces representing monomials
of $w(\rho_{24})w(\rho_{45})$ and $\partial w(\rho_{25})$, respectively.\label{Flo:A-algebra-chain-map}}

\end{figure}

The region $R$ can be naturally split into a union of two admissible
disks or half-disks in two ways (see Figure \ref{Flo:A-algebra-chain-map}):
follow either of the strands of $\partial R$ which intersect at $c'$
as far right as possible until they intersect $\partial R$ again.
If such a path does not end on the dividing line, this splitting contributes
the related monomial to $\partial w_{ij}$; otherwise it ends at some
point $x_{k}$ strictly between $x_{i}$ and $x_{j}$ and so it contributes
that monomial to the product $w_{ik}w_{kj}$. Therefore the monomials
in the sum $\partial w_{ij}+\sum w_{ik}w_{kj}$ can be paired together
as the possible splittings of these regions $R$, and since the two
monomials in each pair are equal, the sum must be zero.
\end{proof}
Since $w$ is a chain map which preserves degree, it is an actual
morphism $I_{n}\to A(K^{A})$ in the category of DGAs.

\subsection{The type D algebra}

Let $K^{D}$ be the right half of a simple Legendrian front diagram
divided along a vertical line, with Maslov potential $\mu$. Let $I_{n}^{\mu}$
be the algebra associated to the points on the intersection of $K^{D}$
and the dividing line, again numbered from $x_{1}$ at the top to
$x_{n}$ at the bottom.

\begin{figure}
\begin{centering}
\includegraphics{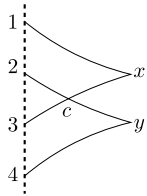}
\par\end{centering}

\caption{A half-diagram $K^{D}$ constructed from the trefoil of Figure \ref{fig:simple-trefoil}.\label{fig:trefoil-d}}

\end{figure}

\begin{defn}
The set $\mathrm{Half}_{D}(K^{D};c)$ consists of all admissible right
half-disks $H$ embedded in $K^{D}$ with rightmost vertex $c$.  These are
defined the same way as admissible disks, but instead of having a unique
leftmost point at a left cusp we require the leftmost part of the boundary
to be a segment of the dividing line between some points $x_{i}$ and $x_{j}$.
We define the word
$\partial H$ to be the product of the following in order: the corners
between $c$ and $x_{i}$ on the boundary of the disk; the element
$\rho_{ij}\in I_{n}$; and then the corners between $x_{j}$ and $c$.
\end{defn}
Note that the set $\mathrm{Disk}(K^{D};c)$ can be defined just as
in the original Chekanov-Eliashberg algebra, so in particular the
left cusp of a disk $D\in\mathrm{Disk}(K^{D};c)$ must lie in the
half-diagram $K^{D}$. 
\begin{defn}
The \emph{type D} algebra $D(K^{D})$ is the DGA generated freely
over $\mathbb{F}$ by the vertices of $K^{D}$ and the generators
$\rho_{ij}$ of $I_{n}$. The cusps have grading $1$ and the crossings
have grading $|c|=\mu(s_{1})-\mu(s_{2})$, where $s_{1}$ and $s_{2}$
are the top and bottom strands through $c$, and the elements $\rho_{ij}$
have grading $\mu(x_{i})-\mu(x_{j})-1$ just as in $I_{n}$.

If $c$ and $c'$ are a crossing and cusp of $K^{D}$, respectively,
then the differential on $D(K^{D})$ is given by the formulas \begin{eqnarray*}
\partial c & = & \sum_{D\in\mathrm{Disk}(K^{D};c)}\partial D+\sum_{H\in\mathrm{Half}_{D}(K^{D};c)}\partial H\\
\partial c' & = & 1+\sum_{D\in\mathrm{Disk}(K^{D};c')}\partial D+\sum_{H\in\mathrm{Half}_{D}(K^{D};c')}\partial H\\
\partial\rho_{ij} & = & \sum_{i<k<j}\rho_{ik}\rho_{kj}.\end{eqnarray*}
\end{defn}
\begin{example}
For $K^{D}$ the half diagram of Figure \ref{fig:trefoil-d}, the
algebra $D(K^{D})$ has generators $x,y,c$ as well as the generators
$\rho_{ij}$, $1\leq i<j\leq4$, of $I_{4}$. The differential on
the vertices is given by \begin{eqnarray*}
\partial x & = & 1+\rho_{12}c+\rho_{13}\\
\partial y & = & 1+\rho_{24}+c\rho_{34}\\
\partial c & = & \rho_{23}.\end{eqnarray*}
\end{example}
\begin{prop}
\label{pro:d-algebra-differential}The differential on $D(K^{D})$
has degree $-1$, and $\partial^{2}=0$.\end{prop}
\begin{proof}
To show that $\deg(\partial)=-1$ we only need to check that $|\partial H|=|v|-1$
for any $H\in\mathrm{Half}_{D}(K^{D};v)$, since it is already true
for full disks $D\in\mathrm{Disk}(K^{D};v)$ as in the case of $Ch(K)$.
Traversing the boundary of $H$ in counterclockwise order from $v$,
we pass through a series of corners $c_{1},\dots,c_{k}$; a segment
connecting two points $x_{a}$ and $x_{b}$ on the dividing line;
and then some more corners $c_{1}',\dots,c_{l}'$ on the way back
to $v$. Since the turn at each corner $c_{i}$ lowers the potential
by $|c_{i}|$, the potential at the top strand $s_{1}$ through $v$
satisfies $\mu(s_{1})-\mu(x_{a})=\sum|c_{i}|$, and likewise $\mu(x_{b})-\mu(s_{2})=\sum|c_{j}'|$
where $s_{2}$ is the bottom strand. Therefore \begin{eqnarray*}
|c|=\mu(s_{1})-\mu(s_{2}) & = & \mu(x_{a})-\mu(x_{b})+\sum_{i=1}^{k}|c_{i}|+\sum_{j=1}^{l}|c_{j}'|\\
 & = & \sum_{i=1}^{k}|c_{i}|+|\rho_{ab}|+\sum_{j=1}^{l}|c_{l}'|+1,\end{eqnarray*}
and since $\partial H=c_{1}\dots c_{k}\rho_{ab}c_{1}'\dots c_{l}'$
we have $|c|-1=|\partial H|$ as desired.

To prove that $\partial^{2}=0$, we proceed as in the proof of Proposition
\ref{pro:w-chain-map}. For a fixed vertex $v$, each monomial in
$\partial^{2}v$ can correspond to a region $R$ with right cusp at
$v$ and two left cusps, but now we need to consider the possibility
that these cusps might lie across the dividing line; in other words,
we may only see the algebra elements $\rho_{ij}$. If the special
vertex $c'$ between the left cusps where $R$ occupies three of four
quadrants appears to the right of the dividing line, then we map split
$R$ in two different ways just as before, by extending either strand
through $c'$ until it hits $\partial R$ again.

\begin{figure}
\centering{}\includegraphics{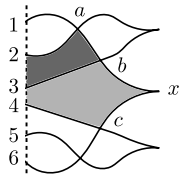}\caption{A region appearing in the proof that $\partial^{2}=0$ for $D(K^{D})$.\label{Flo:d-algebra-differential}}

\end{figure}

The only remaining case is that of a region where the special vertex
$c'$ may be to the left of the dividing line, so that if $K^{D}$ were
completed to a front diagram then $c'$ would be part of the left half $K^{A}$.
In this case $R\subset K^{D}$
is actually a half-disk which intersects the dividing line at some
points $x_{i}$ and $x_{j}$. For any $k$ satisfying $i<k<j$, the
strand through $x_{k}$ must intersect $\partial R$ somewhere; otherwise,
following it would lead us to a right cusp in the interior of $R$,
contradicting the assumption that $K^{D}$ is simple. Then this strand
together with $\partial R$ divides $R$ into a union of two half-disks,
one half-disk $H$ with rightmost vertex at $v$ whose monomial $\partial H$
appears as a term of $\partial R$, and one half-disk $H'$ with rightmost
vertex at some corner $v'$ of $\partial H$. The associated monomial
of $\partial^{2}v$ is obtained by replacing the generator $v'$ in
$\partial H$ with the monomial $\partial H'$, resulting in the monomial
$\partial R$ with $\rho_{ik}\rho_{kj}$ in place of $\rho_{ij}$
since each of $\rho_{ik}$ and $\rho_{kj}$ appear in exactly one
of $\partial H$ and $\partial H'$. But this is also the monomial
which we get from $\partial(\partial R)$ by differentiating the $\rho_{ij}$
term and picking out the $\rho_{ik}\rho_{kj}$ term of $\partial\rho_{ij}$,
so these monomials appear in pairs and their sum must be zero.

For example, Figure \ref{Flo:d-algebra-differential} shows such a
region whose associated monomial is $a\rho_{23}\rho_{34}c$ and which
appears twice in $\partial^{2}x$: once from the term $\partial(b\rho_{34}c)$
using the monomial $a\rho_{23}$ of $\partial b$, and once from the
term $\partial(a\rho_{24}c)$ using the monomial $\rho_{23}\rho_{34}$
of $\partial\rho_{24}$.
\end{proof}
Unlike the algebra $A(K^{A})$, this algebra {}``remembers'' the
interaction of disks with the boundary as part of its differential,
so its differential structure is necessarily more complicated. On the other
hand, the inclusion $I_{n}\hookrightarrow D(K^{D})$ is trivially
a chain map of degree 0, since the differential on elements $\rho_{ij}$
is identical in both algebras.

\subsection{The van Kampen theorem}

Let $K$ be a Legendrian front diagram split into a left half $K^{A}$
and a right half $K^{D}$ by a vertical dividing line which intersects
the front in $n$ points, and suppose we have a Maslov potential $\mu$
associated to this front. Then it is easy to see that we have a commutative
diagram of algebras \begin{equation}
\xymatrix{I_{n}\ar[d]_{w}\ar[r] & D(K^{D})\ar[d]^{w'}\\
A(K^{A})\ar[r] & Ch(K)}
\label{eq:pushout-diagram}\end{equation}
where $I_{n}\to D(K^{D})$ and $A(K^{A})\to Ch(K)$ are inclusion
maps and $w:I_{n}\to A(K^{A})$ is the map defined in section \ref{sec:type-A-algebra},
and the map $w':D(K^{D})\to Ch(K)$ sends vertices to themselves and
elements $\rho_{ij}$ to $w(\rho_{ij})\in A(K^{A})\subset Ch(K)$. 
\begin{lem}
\label{lem:w'-chain-map}The map $w':D(K^{D})\to Ch(K)$ is a chain
map of degree zero, and so the diagram above is a commutative diagram
of DGAs.\end{lem}
\begin{proof}
Clearly $w'$ preserves the degrees of vertices of $K^{D}$, and it
does the same for generators $\rho_{ij}$ by Lemma \ref{lem:w-degree-zero},
so $w'$ has degree zero.

For a generator $\rho_{ij}\in D(K^{D})$ we have $\partial(w'(\rho_{ij}))=\partial(w(\rho_{ij}))=w(\partial\rho_{ij})=w'(\partial\rho_{ij})$
since $w$ is a chain map. If instead we consider a vertex $v\in D(K^{D})$,
then (letting $\epsilon$ be $0$ if $v$ is a crossing and $1$ if
$v$ is a cusp) \begin{eqnarray*}
\partial(w'(v)) & = & \epsilon+\sum_{D\in\mathrm{Disk}(K;v)}\partial D\\
 & = & \epsilon+\sum_{D\in\mathrm{Disk}(K^{D};v)}\partial D+\sum_{i<j}\sum_{\substack{D\in\mathrm{Disk}(K;v)\\
x_{i},x_{j}\in\partial D}
}\partial D.\end{eqnarray*}
The disks $D$ with $x_{i},x_{j}\in\partial D$ can all be obtained
by gluing together a half-disk $H\in\mathrm{Half}_{D}(K^{D};v)$ and
another half-disk $H'\in\mathrm{Half}_{A}(K^{A};i,j)$, and all such
gluings give admissible disks, so \begin{eqnarray*}
\sum_{\substack{D\in\mathrm{Disk}(K;v)\\
x_{i},x_{j}\in\partial D}
}\partial D & = & \sum_{\substack{H\in\mathrm{Half}_{D}(K^{D};v)\\
x_{i},x_{j}\in\partial H}
}\sum_{H'\in\mathrm{Half}_{A}(K^{A};i,j)}\partial(H\cup H')\\
 & = & \sum_{\substack{H\in\mathrm{Half}_{D}(K^{D};v)\\
x_{i},x_{j}\in\partial H}
}\partial H|_{\rho_{ij}=w'(\rho_{ij})}\\
 & = & \sum_{\substack{H\in\mathrm{Half}_{D}(K^{D};v)\\
x_{i},x_{j}\in\partial H}
}w'(\partial H)\end{eqnarray*}
where the notation in the second line means that we have replaced
the unique instance of $\rho_{ij}$ in the monomial $\partial H$
with the expression $w'(\rho_{ij})$. But now \[
\partial(w'(v))=\epsilon+\sum_{D\in\mathrm{Disk}(K^{D};v)}w'(\partial D)+\sum_{i<j}\sum_{\substack{H\in\mathrm{Half}_{D}(K^{D};v)\\
x_{i},x_{j}\in\partial H}
}w'(\partial H)=w'(\partial v)\]
and so $w'$ is a chain map as desired.
\end{proof}
\begin{defn}
Let $A\stackrel{f}{\to}B$ and $A\stackrel{g}{\to}C$ be morphisms in some
category.  Suppose there is an object $D$ together with morphisms
$B\stackrel{h}{\to}D$ and $C\stackrel{i}{\to}D$ such that $h\circ f=i\circ g$.
Then $(D,h,i)$ is said to be the {\em pushout} of $f$ and $g$ if it satisfies the following universal property: for every commutative diagram \[
\xymatrix{A\ar[d]_{g}\ar[r]^{f} & B\ar[d]^{h}\ar[ddr]\\
C\ar[r]^{i}\ar[rrd] & D\\
 &  & X}
\]
there exists a unique morphism $D\to X$ making the diagram commute.
\end{defn}
Now that we have expended considerable effort to construct the commutative
diagram (\ref{eq:pushout-diagram}), the following theorem is an easy
consequence.
\begin{thm}
\label{thm:pushout-square}This diagram is a pushout square in the
category of DGAs.\end{thm}
\begin{proof}
Suppose we have another commutative diagram of DGAs as follows: \[
\xymatrix{I_{n}\ar[d]_{w}\ar[r] & D(K^{D})\ar[d]^{w'}\ar[ddr]^{g}\\
A(K^{A})\ar[r]\ar[rrd]_{f} & Ch(K)\ar@{-->}[dr]^{\varphi}\\
 &  & X}
\]
Then it is easy to construct the dotted morphism $\varphi:Ch(K)\to X$.
The algebra $Ch(K)$ is generated by vertices of the front diagram
$K$; if a vertex $v$ is on the left side of the dividing line, then
it is in the diagram $K^{A}$ and we let $\varphi(v)=f(v)$, and otherwise
it is in $K^{D}$ and we let $\varphi(v)=g(v)$. This is clearly well-defined
and makes the diagram commute, so if it is a chain map (i.e. a morphism
of DGAs) then $Ch(K)$ has the universal property of a pushout.

For $v\in A(K^{A})\subset Ch(K)$ we have $\partial(\varphi(v))=\partial(f(v))=f(\partial v)=\varphi(\partial v)$,
since $v\in A(K^{A})$ implies that $\partial v\in A(K^{A})\subset Ch(K)$
as well. On the other hand, for $v\in Ch(K)$ coming from the $K^{D}$
side of the diagram and $v_{D}$ the corresponding generator of $D(K^{D})$
we have $\partial(\varphi(v))=\partial(\varphi\circ w'(v_{D}))=\partial(g(v_{D}))=g(\partial v_{D})$
since $g$ is a chain map, and then $g(\partial v_{D})=\varphi(w'(\partial v_{D}))=\varphi(\partial(w'(v_{D})))=\varphi(\partial v)$
since $w'$ is also a chain map by Lemma \ref{lem:w'-chain-map} and
so $\partial(\varphi(v))=\varphi(\partial v)$ in this case as well.
Therefore $\varphi$ is a chain map, as desired.\end{proof}
\begin{rem}
The result $Ch(K)=A(K^{A})\coprod_{I_{n}}D(K^{D})$ of Theorem \ref{thm:pushout-square}
is a noncommutative analogue of the pairing theorem $CP^{-}(\mathcal{H})\cong CPA^{-}(\mathcal{H}^{A})\otimes_{\mathcal{A}_{N,k}}CPD^{-}(\mathcal{H}^{D})$
of \cite{Lipshitz:2008p652}; even the construction of $D(K^{D})$
as an algebra of the form $I_{n}\coprod\mathbb{F}\langle v_{i}\rangle$,
where the $v_{i}$ are the vertices of $K^{D}$, can be compared to
the definition $CPD^{-}(\mathcal{H}^{D})=\mathcal{A}_{N,k}\otimes_{I_{N,k}}\mathbb{A}\langle\mathfrak{S}(\mathcal{H}^{D})\rangle$.
Theorem \ref{thm:pushout-square} originated as an attempt to adapt
the pairing theorem for $CP^{-}$ to the Chekanov-Eliashberg algebra,
since both $Ch(K)$ and the non-invariant $CP^{-}$ are defined in
terms of embedded disks in the plane rather than in the torus of combinatorial
knot Floer homology.
\end{rem}

\subsection{Type DA algebras and the generalized van Kampen theorem}

\begin{figure}

\begin{centering}
\includegraphics{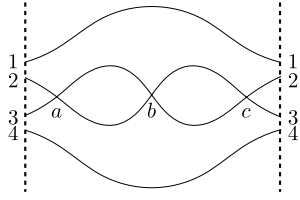}
\par\end{centering}

\caption{A Legendrian front diagram with two dividing lines.\label{fig:trefoil-da}}

\end{figure}

Suppose we want to divide a simple Legendrian front into multiple
pieces along vertical lines, as in the bordered front $K$ of Figure
\ref{fig:trefoil-da}. We can associate a so-called DGA of \emph{type
DA} to $K$ generalizing both the type A and type D algebras, and
the analogue of the pairing theorem will follow with minimal effort. 
\begin{defn}
The algebra $DA(K)$ is the DGA generated freely over $\mathbb{F}$
by the vertices of $K$ and the generators of the algebra $I_{n}$
corresponding to the left dividing line. The grading and differential
on $DA(K)$ are defined exactly as in the type D algebra.
\end{defn}
In Figure \ref{fig:trefoil-da}, for example, $DA(K)$ is generated
by $a,b,c$ and the elements $\rho_{ij}\in I_{4}$ with $1\leq i<j\leq4$.
The differential is given by $\partial a=\rho_{23}$, $\partial b=\partial c=0$,
and $\partial\rho_{ij}=\sum_{i<k<j}\rho_{ik}\rho_{kj}$.
\begin{lem}
The differential on $DA(K)$ has degree $-1$ and satisfies $\partial^{2}=0$.\end{lem}
\begin{proof}
We repeat the proof of Proposition \ref{pro:d-algebra-differential}
word for word, replacing $D(K^{D})$ with $DA(K)$ as needed.
\end{proof}
Let $I_{m}'$ be the algebra corresponding to the right dividing line,
with generators denoted $\rho_{ij}'$. Then we can define the set
of half-disks $\mathrm{Half}_{DA}(K;i,j)$ almost as in Definition
\ref{def:half-disks-a}: the right boundary of a half-disk $H$ should
still be the segment between points $x_{i}'$ and $x_{j}'$ on the
right dividing line, but now the left boundary is allowed to be a
segment on the left dividing line connecting some points $x_{k}$
and $x_{l}$, in which case the monomial $\partial H$ contains the
generator $\rho_{kl}$ in the appropriate place.
\begin{defn}
Define an algebra homomorphism $w:I_{m}'\to DA(K)$ by the formula
\[
w(\rho_{ij}')=\sum_{H\in\mathrm{Half}_{DA}(K;i,j)}\partial H.\]

\end{defn}
For example, the map $w:I_{4}'\to DA(K)$ in Figure \ref{fig:trefoil-da}
is given by:

\begin{center}
\begin{tabular}{ll}
$w(\rho_{12}')=\rho_{12}(abc+a+c)+\rho_{13}(bc+1)$ & $w(\rho_{13}')=\rho_{12}(ab+1)+\rho_{13}b$\tabularnewline
$w(\rho_{14}')=\rho_{14}$ & $w(\rho_{23}')=0$\tabularnewline
$w(\rho_{34}')=(cb+1)\rho_{24}+(cba+c+a)\rho_{34}$ & $w(\rho_{24}')=b\rho_{24}+(ba+1)\rho_{34}$\tabularnewline
\end{tabular}
\par\end{center}
\begin{prop}
The map $w:I_{m}'\to DA(K)$ is a morphism of DGAs.\end{prop}
\begin{proof}
See the proofs of Lemma \ref{lem:w-degree-zero} and Proposition \ref{pro:w-chain-map},
with some minor changes as in the proof of Proposition \ref{pro:d-algebra-differential}
to account for the differentials of each $\rho_{kl}$ that might appear
in $w(\rho_{ij}')$.
\end{proof}
The type DA algebra generalizes both the type D algebra, by incorporating
the algebra $I_{n}$ of the left dividing line into the DGA structure,
and the type A algebra, by admitting an analogous morphism from $I_{m}'$
for the right dividing line. In fact, both the type A and type D algebras
are special cases of this, with $n=0$ and $m=0$ respectively.

\begin{figure}
\begin{centering}
\includegraphics{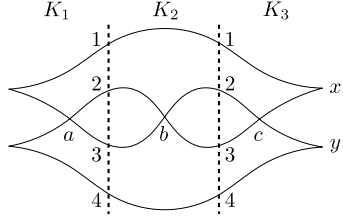}
\par\end{centering}

\caption{A trefoil diagram split into three regions by a pair of dividing lines.\label{fig:trefoil-a-da-d}}

\end{figure}

We can use this more general structure to relate overlapping pieces
of a simple Legendrian front. Consider three regions $K_{1}$, $K_{2}$,
and $K_{3}$ of a simple front as in Figure \ref{fig:trefoil-a-da-d},
and let $K_{12}$, $K_{23}$ and $K_{123}$ denote the larger regions
$K_{1}\cup K_{2}$, $K_{2}\cup K_{3}$, and $K_{1}\cup K_{2}\cup K_{3}$
respectively. Then the map $w:DA(K_{2})\to DA(K_{12})$ which preserves
the vertices of $K_{2}$ and sends $\rho_{ij}$ to the appropriate
element $w(\rho_{ij})\in DA(K_{1})\subset DA(K_{12})$ is a chain
map, as are the inclusion $DA(K_{2})\hookrightarrow DA(K_{23})$ and
the map $w':DA(K_{23})\to DA(K_{123})$. The proofs of these facts
proceed exactly as expected, as does the following theorem.
\begin{thm}
\label{thm:generalized-pairing}The commutative diagram \[
\xymatrix{DA(K_{2})\ar[r]\ar[d]_{w} & DA(K_{23})\ar[d]^{w'}\\
DA(K_{12})\ar[r] & DA(K_{123})}
\]
is a pushout square in the category of DGAs.
\end{thm}
In the special case where $K_{2}$ is a product cobordism, so both
dividing lines have the same number of points and each strand in $K_{2}$
connects $x_{i}$ to $x_{i}'$ without any crossings or cusps, then
the inclusion $I_{n}\hookrightarrow DA(K_{2})$ of the left dividing
line of $K_{2}$ is an isomorphism and so is the chain map $w:I_{n}'\to DA(K_{2})$
coming from the right dividing line (i.e. $w(\rho_{ij}')=\rho_{ij}$).
If furthermore the regions $K_{1}$ and $K_{3}$ have no left and
right dividing lines, respectively, so that $DA(K_{1})=A(K_{1})$
and $DA(K_{3})=D(K_{3})$, then $DA(K_{123})=Ch(K)$ and Theorem \ref{thm:generalized-pairing}
reduces to the statement of Theorem \ref{thm:pushout-square}.

\section{Augmentations}

Since it can be hard to distinguish between Legendrian knots given
only a presentation of their algebras, Chekanov introduced the notion
of linearization.
\begin{defn}
An \emph{augmentation} of a DGA is a morphism $\epsilon:A\to\mathbb{F}$,
where $\mathbb{F}$ is concentrated in degree $0$ and has vanishing
differential. In particular we require $\epsilon\circ\partial=0$,
$\epsilon(1)=1$, and $\epsilon(x)=0$ for any element $x$ of pure
nonzero degree.
\end{defn}
Given an augmentation $\epsilon$ of the algebra $A$ freely generated
by a finite set of elements $\{v_{i}\}$, the differential on $A$
turns the $\mathbb{F}$-vector space $A^{\epsilon}=\ker(\epsilon)/(\ker(\epsilon))^{2}$
with basis $\{v_{i}-\epsilon(v_{i})\}$ into a chain complex. We can
then compute the associated Poincar\'{e} polynomial $P_{\epsilon}(t)=\sum_{\lambda\in\Gamma}\dim(H_{\lambda}(A^{\epsilon}))t^{\lambda}$.
\begin{thm}[{\cite[Theorem 5.2]{Chekanov:2002p539}}]
 The set of \emph{Chekanov polynomials} $\{P_{\epsilon}(t)\mid\epsilon\mbox{ an augmentation of }Ch(K)\}$
is invariant under stable tame isomorphisms of $Ch(K)$ and is therefore
a Legendrian isotopy invariant.
\end{thm}
It is possible for $Ch(K)$ to have multiple augmentations giving
the same polynomial $P_{\epsilon}(t)$. Melvin and Shrestha \cite{Melvin:2005p636}
constructed prime Legendrian knots with arbitrarily many Chekanov
polynomials and also showed that every Laurent polynomial of the form
$P(t)=t+p(t)+p(t^{-1})$ (i.e. those satisfying Sabloff's duality
theorem \cite{Sabloff:2006p526}), with $p(t)$ a polynomial with
positive integer coefficients, is a Chekanov polynomial of some knot.
On the other hand, not every Legendrian knot even admits a single
augmentation; the existence of augmentations is known to be equivalent
to the existence of a normal ruling \cite{Fuchs:2003p674,Fuchs:2004p673,Sabloff:2005p671},
which implies, for example, that $K$ must have rotation number $0$
\cite[Theorem 1.3]{Sabloff:2005p671}.

\subsection{A Mayer-Vietoris sequence for linearized homology}

Suppose that the simple Legendrian front $K$ is divided by a vertical
line into left and right halves $K^{A}$ and $K^{D}$. By Theorem
\ref{thm:pushout-square}, an augmentation $\epsilon$ of $Ch(K)$
is equivalent to a commutative diagram \[
\xymatrix{I_{n}\ar[r]\ar[d]_{w} & D(K^{D})\ar[d]^{\epsilon_{D}}\\
A(K^{A})\ar[r]^{\epsilon_{A}} & \mathbb{F}}
\]
of DGAs, in which case $\epsilon_{A}$ and $\epsilon_{D}$ both factor
through $Ch(K)$. We can associate a Mayer-Vietoris sequence to the
associated linearizations, denoted $I^{\epsilon}$, $A^{\epsilon}$,
$D^{\epsilon}$, and $Ch^{\epsilon}$.
\begin{thm}
\label{thm:mayer-vietoris}There is a long exact sequence \[
\dots\to H_{k}(I^{\epsilon})\to H_{k}(A^{\epsilon})\oplus H_{k}(D^{\epsilon})\to H_{k}(Ch^{\epsilon})\to H_{k-1}(I^{\epsilon})\to\dots\]
of linearized homology groups.\end{thm}
\begin{proof}
It suffices to show that the sequence \[
0\to I^{\epsilon}\stackrel{f}{\to}A^{\epsilon}\oplus D^{\epsilon}\stackrel{g}{\to}Ch^{\epsilon}\to0\]
of chain complexes is exact, where $f(x)=(-w(x),x)$ and $g(x,y)=x+w'(y)$.
Here we abuse notation and let $w$ and $w'$ refer to the linearized maps
$I^{\epsilon}\to A^{\epsilon}$ and $D^{\epsilon}\to Ch^{\epsilon}$ induced by
$w:I_{n}\to A(K^{A})$ and $w':D(K^{D})\to Ch(K)$.

Clearly $f$ is injective, since $I^{\epsilon}\to D^{\epsilon}$ is
an inclusion map, and $g$ is surjective since any generator $v-\epsilon(v)$
of $Ch^{\epsilon}$ is the image of either $(v-\epsilon(v),0)$ or
$(0,v-\epsilon(v))$ depending on whether $v$ is a vertex in $K^{A}$
or $K^{D}$.

To see that $\mathrm{im}(f)\subset\ker(g)$, or equivalently that
$g\circ f=0$, consider a generator $\rho_{ij}-\epsilon(\rho_{ij})$
of $I^{\epsilon}$. We can compute $f(\rho)=(-w(\rho_{ij})+\epsilon(\rho_{ij}),\rho_{ij}-\epsilon(\rho_{ij}))$,
and so \[
g\circ f(\rho_{ij}-\epsilon(\rho_{ij}))=(-w(\rho_{ij})+\epsilon(\rho_{ij}))+(w'(\rho_{ij})-\epsilon(\rho_{ij}))=0.\]
Conversely, given $(x,y)\in\ker(g)$, we have $x+w'(y)=0$ in $Ch^{\epsilon}$.
If we write $y$ as a sum of generators $\rho_{ij}-\epsilon(\rho_{ij})$
and $v_{k}-\epsilon(v_{k})$ of $D^{\epsilon}$, where the $v_{k}$
are vertices of $K^{D}$, then $y$ must not include any of the latter
terms since they cannot be eliminated by any $x$ in the subcomplex
$A^{\epsilon}\subset Ch^{\epsilon}$. But then $y\in D^{\epsilon}$
is the image of some element $\rho\in I^{\epsilon}$ under the inclusion
$I^{\epsilon}\hookrightarrow D^{\epsilon}$, and since $x+w'(y)=0$
we have $x=-w(\rho)$, hence $(x,y)=f(\rho)\subset\mathrm{im}(f)$.\end{proof}
\begin{rem}
Given a pushout $DA(K_{123})=DA(K_{12})\coprod_{DA(K_{2})}DA(K_{23})$
as in Theorem \ref{thm:generalized-pairing} and augmentations
$\epsilon_{12}$ and $\epsilon_{23}$ which agree on $DA(K_{2})$, we get
another augmentation
$\epsilon:DA(K_{123})\to\mathbb{F}$; then an identical argument gives
an analogous long exact sequence \[
\dots\to H_{k}(K_{2}^{\epsilon})\to H_{k}(K_{12}^{\epsilon})\oplus H_{k}(K_{23}^{\epsilon})\to H_{k}(K_{123}^{\epsilon})\to H_{k-1}(K_{2}^{\epsilon})\to\dots.\]

\end{rem}

\subsection{Connected sums}

In some simple cases we can use the long exact sequence to explicitly
work out the linearizations of some type A and type D algebras, reproving
a result about the homology of connected sums which appeared in \cite{Chekanov:2002p539,Melvin:2005p636}.

\begin{figure}
\centering{}\includegraphics{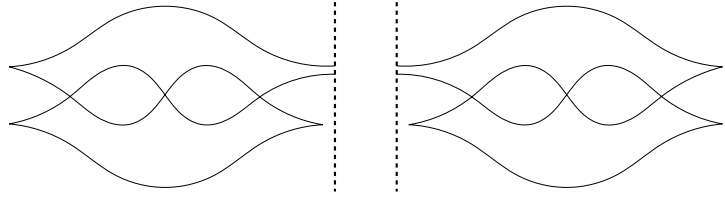}\caption{Diagrams $K^{A}$ and $K^{D}$ created by removing a single right
or left cusp from a front $K$.\label{fig:connect-sum}}

\end{figure}

\begin{example}
\label{exa:connect-sum-A}Let $K^{A}$ be the left half of a diagram
constructed by removing a single right cusp $x$ from $K$ as in the
left side of Figure \ref{fig:connect-sum}. Let $\rho$ be the generator
of $I_{2}$; then the map $w:I_{2}\to A(K^{A})$ sends $\rho$ to
$\partial x-1$. An augmentation $\epsilon$ of $Ch(K)$ then immediately
gives an augmentation $\epsilon_{A}$ of $A(K^{A})$, and since $\epsilon(\partial x)=0$
we must have $\epsilon_{A}(w(\rho))=1$.

The algebra $D$ corresponding to the right half of $K$ is generated
by the right cusp $x$ and the generator $\rho\in I_{2}$, with $|x|=1$,
$|\rho|=0$, $\partial x=\rho+1$, and $\partial\rho=0$. An augmentation
$\epsilon_{D}$ of $D$ must satisfy $\epsilon_{D}(\rho)=1$ since
$\epsilon_{D}(\partial x)=0$, so the linearization $D^{\epsilon}$
is generated by $x$ and $\rho+1$ with $\partial x=\rho+1$ and thus
its homology is zero. On the other hand, the corresponding augmentation
of $I_{2}$ has homology $\langle\rho+1\rangle\cong\mathbb{F}$ in
degree zero. Now by Theorem \ref{thm:mayer-vietoris}, the exact sequence
\[
H_{k}(I_{2}^{\epsilon})\to H_{k}(A^{\epsilon})\oplus H_{k}(D^{\epsilon})\to H_{k}(Ch(K)^{\epsilon})\to H_{k-1}(I_{2}^{\epsilon})\]
gives an isomorphism $H_{k}(Ch(K))\cong H_{k}(A^{\epsilon})\oplus H_{k}(D^{\epsilon})\cong H_{k}(A^{\epsilon})$
when $k\not=0,1$. Otherwise we have an exact sequence \[
0\to H_{1}(A^{\epsilon})\to H_{1}(Ch(K)^{\epsilon})\to\langle\rho+1\rangle\stackrel{w_{*}}{\to}H_{0}(A^{\epsilon})\to H_{0}(Ch(K)^{\epsilon})\to0.\]

Sabloff proved in \cite[Section 5]{Sabloff:2006p526} that $Ch(K)^{\epsilon}$
has a fundamental class $[\kappa]$ which is nonzero in $H_{1}(Ch(K)^{\epsilon})$,
where $\kappa=\sum_{v\in V}v$ for some subset $V$ of the vertices
of $K$ which includes all of the right cusps; in particular $x\in V$.
But then $\partial\kappa=0$ implies $\sum_{v\in V}\partial v=0$,
and so \[
w(\rho+1)=\partial x=\sum_{v\in V\backslash\{x\}}\partial v=\partial\left(\sum_{v\in V\backslash\{x\}}v\right).\]
The right hand side is well-defined and trivial in $H_{0}(A^{\epsilon})$
since every vertex of $V\backslash\{x\}$ is in $K^{A}$, so $w_{*}[\rho+1]=0$.
But applying this to the exact sequence above gives $H_{0}(A^{\epsilon})\cong H_{0}(Ch(K)^{\epsilon})$
and $H_{1}(Ch(K)^{\epsilon})\cong H_{1}(A^{\epsilon})\oplus\mathbb{F}$,
so \[
P_{\epsilon_{A}}^{K^{A}}(t)=P_{\epsilon}^{K}(t)-t.\]

\end{example}
We note here that removing a right cusp has changed the linearized
homology by removing the fundamental class, just as removing a point
from a manifold will eliminate its fundamental class. We speculate
in general that given a half-diagram, one might be able to define
an appropriate notion of compactly supported homology which does not
count disks approaching the dividing line, and use this to recover
a notion of Poincar\'{e} duality analogous to $H_{k}(M^{n})\cong H_{c}^{n-k}(M^{n})$.
\begin{example}
\label{exa:connect-sum-D}Let $K^{D}$ be constructed by removing
a single left cusp from $K$ as in the right side of Figure \ref{fig:connect-sum}.
Then $D(K^{D})$ is generated by the vertices of $Ch(K)$ plus the
generator $\rho$ of $I_{2}$, and an augmentation $\epsilon$ of
$Ch(K)$ gives an augmentation of $D(K^{D})$ with $\epsilon_{D}(\rho)=1$.
The algebra $A$ corresponding to the left half of $K$ has no generators,
since there is only a left cusp, and the map $w:I_{2}\to A$ is given
by $w(\rho)=1$. As in Example \ref{exa:connect-sum-A}, the long
exact sequence on homology gives $H_{k}(Ch(K)^{\epsilon})\cong H_{k}(D^{\epsilon})$
for $k\not=0,1$, and then we have an exact sequence \[
0\to H_{1}(D^{\epsilon})\to H_{1}(Ch(K)^{\epsilon})\to\langle\rho+1\rangle\stackrel{i_{*}}{\to}H_{0}(D^{\epsilon})\to H_{0}(Ch(K)^{\epsilon})\to0\]
where $i:I_{2}^{\epsilon}\to D^{\epsilon}$ is the inclusion map.
Therefore \[
P_{\epsilon_{D}}^{K^{D}}(t)=\begin{cases}
P_{\epsilon}^{K}(t)-t, & i_{*}[\rho+1]=0\\
P_{\epsilon}^{K}(t)+1, & i_{*}[\rho+1]\not=0.\end{cases}\]
We conjecture that only the first case occurs, as in Example \ref{exa:connect-sum-A}.\end{example}
\begin{prop}
Let $\epsilon_{1}$ and $\epsilon_{2}$ be augmentations of knots
$K_{1}$ and $K_{2}$. Then their connected sum $K=K_{1}\#K_{2}$,
formed by removing a right cusp from $K_{1}$ and a left cusp from
$K_{2}$ and gluing them together as in Figure \ref{fig:connect-sum},
has a canonical augmentation $\epsilon$ with Chekanov polynomial
$P_{\epsilon}^{K}(t)=P_{\epsilon_{1}}^{K_{1}}(t)+P_{\epsilon_{2}}^{K_{2}}(t)-t$.\end{prop}
\begin{proof}
Removing cusps from $K_{1}$ and $K_{2}$ as described, and assigning
Maslov potentials so that the strands at each removed cusp have matching
potentials, gives half diagrams $K^{A}$ and $K^{D}$ with augmentations
$\epsilon_{A}:A(K^{A})\to\mathbb{F}$ and $\epsilon_{D}:D(K^{D})\to\mathbb{F}$
as in Examples \ref{exa:connect-sum-A} and \ref{exa:connect-sum-D}.
Since these satisfy $\epsilon_{A}(w(\rho))=1$ and $\epsilon_{D}(\rho)=1$,
they are compatible with the maps $w:I_{2}\to A(K^{A})$ and $i:I_{2}\hookrightarrow D(K^{D})$
and thus give an augmentation $\epsilon:Ch(K)\to\mathbb{F}$ by Theorem
\ref{thm:pushout-square}.

Once again $H_{k}(I_{2}^{\epsilon})$ vanishes for $k\not=0$, so
$H_{k}(Ch(K)^{\epsilon})\cong H_{k}(A^{\epsilon})\oplus H_{k}(D^{\epsilon})$
for $k\not=0,1$, and we have an exact sequence \[
0\to H_{1}(A^{\epsilon})\oplus H_{1}(D^{\epsilon})\to H_{1}(Ch^{\epsilon})\to\langle\rho+1\rangle\stackrel{f}{\to}H_{0}(A^{\epsilon})\oplus H_{0}(D^{\epsilon})\to H_{0}(Ch^{\epsilon})\to0.\]
Recalling that $w_{*}:\langle\rho+1\rangle\to H_{0}(A^{\epsilon})$
is zero, this leaves us with two cases depending on the image of the
map $f$: \[
P_{\epsilon}^{K}(t)=\begin{cases}
(P_{\epsilon_{1}}^{K_{1}}(t)-t)+(P_{\epsilon_{2}}^{K_{2}}(t)-t)+t, & f([\rho+1])=(0,0)\\
(P_{\epsilon_{1}}^{K_{1}}(t)-t)+(P_{\epsilon_{2}}^{K_{2}}(t)+1)-1, & f([\rho+1])=(0,y)\end{cases}\]
where $y$ is some nonzero homology class. In both cases this simplifies
to $P_{\epsilon}^{K}(t)=P_{\epsilon_{1}}^{K_{1}}(t)+P_{\epsilon_{2}}^{K_{2}}(t)-t$,
as desired.
\end{proof}

\section{Tangle replacement}

Suppose we want to consider the effect of a tangle replacement on
the DGA of a front. We can try to isolate the tangle by placing dividing
lines on either side, comparing the type DA algebras of the corresponding
section of the diagram both before and after the replacement, and
applying Theorem \ref{thm:generalized-pairing}. This is hard in general
because in addition to comparing the type DA algebras, we must ensure
that the algebras of both dividing lines act compatibly on the type
DA algebras.

We can avoid this problem almost completely by applying a trick from
\cite[Chapter 5]{Ng:2001p667}. Given a tangle $T$ in the middle
of the diagram, we can perform a series of Legendrian Reidemeister
moves to lift it to the top of the diagram and then pull it to the
right end of the front by an isotopy:

\begin{center}
\includegraphics{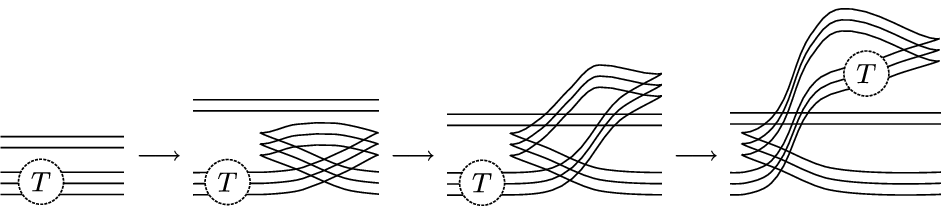}
\par\end{center}

\noindent The effect of the replacement on this new front is often
much easier to determine.
\begin{prop}
\label{pro:tangle-d-diagram}Let $T_{1}$ and $T_{2}$ be Legendrian
tangles, and let $\tilde{T}_{1}$ and $\tilde{T}_{2}$ be half-diagrams
constructed from $T_{1}$ and $T_{2}$ as in Figure \ref{fig:tangle-d-diagram},
possibly modified by some Legendrian Reidemeister moves.  Then
given a morphism $\varphi:D(\tilde{T}_{1})\to D(\tilde{T}_{2})$ which
fixes $\rho_{ij}$ for all $i$ and $j$, we have a pushout diagram
\[
\xymatrix{D(\tilde{T}_{1})\ar[r]^{\varphi}\ar[d] & D(\tilde{T}_{2})\ar[d]\\
Ch(K_{1})\ar[r]^{\tilde{\varphi}} & Ch(K_{2})}
\]
where $K_{1}$ and $K_{2}$ are any fronts which differ only by replacing
$T_{1}$ with $T_{2}$.

\begin{figure}
\begin{centering}
\includegraphics{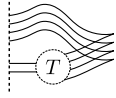}
\par\end{centering}

\caption{The half-diagram $\tilde{T}$ associated to a tangle $T$ in Proposition
\ref{pro:tangle-d-diagram}.\label{fig:tangle-d-diagram}}

\end{figure}
\end{prop}
\begin{proof}
Use the trick mentioned above to modify each front $K_{i}$ by producing
$\tilde{T}_{i}$ on the right side of the diagram, and
place a dividing line in each $K_{i}$ which separates $\tilde{T}_{i}$
from some half-diagram $K^{A}$ on the left; then $K^{A}$
is independent of $i$, as is the map $w:I_{n}\to A(K^{A})$. Consider
the commutative diagram \[
\xymatrix{I_{n}\ar[r]\ar[d]_{w} & D(\tilde{T}_{1})\ar[r]^{\varphi}\ar[d] & D(\tilde{T}_{2})\ar[d]\\
A(K^{A})\ar[r] & Ch(K_{1})\ar[r] & \mathcal{A}}
\]
where the left square is a pushout by Theorem \ref{thm:pushout-square}
and $\mathcal{A}$ is some DGA making the right square a pushout as
well. Since pushouts are associative, the outer rectangle of this
diagram is a pushout square as well, so $\mathcal{A}$ must be isomorphic
to $Ch(K_{2})$ and the right square gives the desired diagram.\end{proof}
\begin{rem}
In general, the algebras $Ch(K_{1})$ and $Ch(K_{2})$ of fronts in
which we perform tangle replacements are not identical to the ones
for which we have the morphism $\tilde{\varphi}$, since $\tilde{\varphi}$
is constructed from equivalent fronts in which we can isolate the
half-diagrams $D(\tilde{T}_{i})$, but they are the same up to stable
tame isomorphism. Thus in applications we will write $Ch(K)$ to refer to a
DGA which is stable tame isomorphic to $Ch(K)$, but this should not cause
any confusion.\end{rem}
\begin{cor}
Let $K_{1}$ and $K_{2}$ differ by replacing tangle $T_{1}$ with
$T_{2}$, and suppose we have a morphism $\varphi:D(\tilde{T}_{1})\to D(\tilde{T}_{2})$
as in Proposition \ref{pro:tangle-d-diagram}. If $Ch(K_{2})$ admits
an augmentation, then so does $Ch(K_{1})$.\end{cor}
\begin{proof}
By Proposition \ref{pro:tangle-d-diagram} we have a morphism $\tilde{\varphi}:Ch(K_{1})\to Ch(K_{2})$,
and since an augmentation of $Ch(K_{2})$ is just a morphism $\epsilon:Ch(K_{2})\to\mathbb{F}$
it follows that $\epsilon\circ\tilde{\varphi}$ is an augmentation
of $Ch(K_{1})$.
\end{proof}
In the following subsections we will give several applications of
this result.  We will adopt the convention that a double arrow in any
figure refers to a tangle replacement or other move which changes the 
Legendrian knot or tangle in question, whereas a single arrow indicates a
Legendrian isotopy.

\subsection{Breaking a pair of horizontal strands}

Consider the effect of the following tangle replacement:

\noindent \begin{center}
\includegraphics{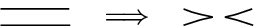}
\par\end{center}

\noindent where in both tangles, the upper strands have Maslov potential
$\mu+1$ and the lower strands have Maslov potential $\mu$ for some $\mu$.
We will label the left tangle consisting of two parallel strands by $P$,
and the right tangle consisting of two cusps by $C$.
Construct the half-diagrams $\tilde{P}$ and $\tilde{C}$ as follows:

\begin{center}
\includegraphics{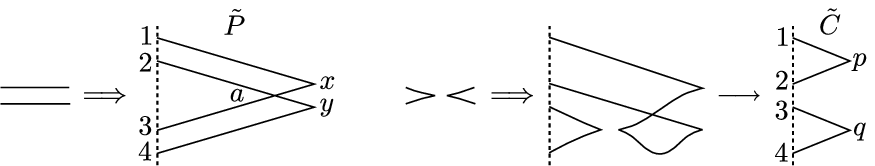}
\par\end{center}

\noindent where in both $\tilde{P}$ and $\tilde{C}$, the strands
through points $1,2,3,$ and $4$ on the dividing line have potentials
$\mu+2,\mu+1,\mu+1,$ and $\mu$ respectively.

Construct a new DGA $D'$ by adding an extra free generator $c$ to
the type D algebra $D(\tilde{P})$ satisfying $\partial c=1+\rho_{12}$.
Then $D'$ is generated by $c,a,x,y,$ and $\rho_{ij}$ for $1\leq i<j\leq4$
satisfying \begin{eqnarray*}
\partial c & = & 1+\rho_{12}\\
\partial a & = & \rho_{23}\\
\partial x & = & 1+\rho_{12}a+\rho_{13}\\
\partial y & = & 1+a\rho_{34}+\rho_{24}\end{eqnarray*}
with gradings $|a|=0$ and $|c|=|x|=|y|=1$. On the other hand, the
algebra $D(\tilde{C})$ is generated by $p,q$ and $\rho_{ij}$ with
$\partial p=1+\rho_{12}$ and $\partial q=1+\rho_{34}$, and $|p|=|q|=1$.
(In both algebras we have $|\rho_{14}|=1$, $|\rho_{23}|=-1$, and
$|\rho_{ij}|=0$ for all other $\rho_{ij}$.)
\begin{lem}
The algebra $D'$ is stable tame isomorphic to $D(\tilde{C})$ by
isomorphisms fixing all of the generators $\rho_{ij}$.\end{lem}
\begin{proof}
Apply a sequence of tame isomorphisms to $D'$ of the form \begin{eqnarray*}
a & \to & a + c\rho_{23} + \rho_{13} + 1 \\
x & \to & x + c(a + c\rho_{23} + c\rho_{13} + 1) \\
y & \to & y + c\rho_{24} + \rho_{14} + x\rho_{34};
\end{eqnarray*}
we can now easily compute that $\partial a=0$, $\partial x=a$, and 
$\partial y=1+\rho_{34}$.  Relabeling $c$ and $y$ by $p$ and $q$, respectively,
and destabilizing to remove the generators $x$ and $a$ sends $D'$ to 
$D(\tilde{C})$, as desired.\end{proof}
\begin{thm}
\label{thm:parallel-break}Let $K'$ be the front obtained from a
Legendrian front $K$ by replacing the tangle $P$ with $C$. Then
$Ch(K)$ and $Ch(K')$ are stable tame isomorphic to DGAs $\mathcal{A}$
and $\mathcal{A}'$, where $\mathcal{A}'$ is obtained from $\mathcal{A}$
by adding a single free generator $c$ in grading 1.  Thus if $Ch(K')$
admins an augmentation, then so does $Ch(K)$.\end{thm}
\begin{proof}
We have constructed an inclusion $D(\tilde{P})\hookrightarrow D'\cong D(\tilde{U})$,
so Proposition \ref{pro:tangle-d-diagram} gives us the induced map
$Ch(K)\hookrightarrow Ch(K')$.
\end{proof}
\begin{rem}
Once we have the morphism $D(\tilde{P})\hookrightarrow D(\tilde{U})$,
we could just construct the map $Ch(K)\to Ch(K')$ directly by using
the same sequence of tame isomorphisms and destabilizations, but replacing
each $\rho_{ij}$ with $w(\rho_{ij})\in A(K^{A})\subset Ch(K)$.
\end{rem}
We can use this to draw similar conclusions about other tangle replacements
as well. For example:
\begin{cor}
\label{cor:half-clasp-augmentation}Let $K'$ be obtained from $K$
by any of the following tangle replacements, where the crossings removed by
each replacement have grading $0$:

\begin{center}
\includegraphics{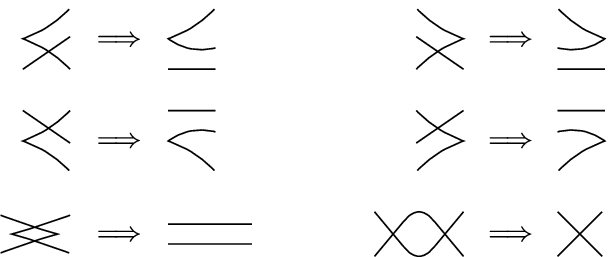}
\par\end{center}

\noindent Then there are DGA maps $Ch(K)\to Ch(K')$, constructed
exactly as in Theorem \ref{thm:parallel-break}, and if $Ch(K')$
has an augmentation then so does $Ch(K).$\end{cor}
\begin{proof}
We prove the first of these by applying Theorem \ref{thm:parallel-break}
to the tangle in a small neighborhood of the dotted line, and then
performing a type I Reidemeister move:

\begin{center}
\includegraphics{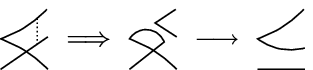}
\par\end{center}

\noindent The proofs of the next three cases are identical, and the
last two are proven as follows:

\begin{center}
\includegraphics{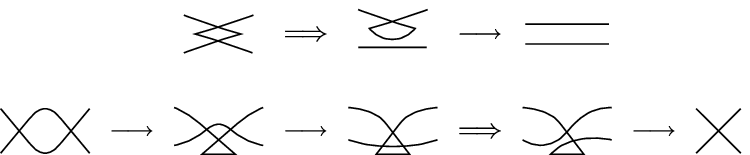}
\par\end{center}

\noindent where in each case we use one of the first four tangle replacements
(indicated by a double arrow) together with some Legendrian Reidemeister
moves.\end{proof}
\begin{example}
Let $K$ be the Legendrian closure of a positive braid in the sense of
\cite{Kalman:2006p630},
so that every crossing has grading $0$. If the top strand of the
braid is not part of any crossing, then it belongs to a Legendrian
unknot (i.e. a topological unknot with $tb=-1$ and $r=0$) disjoint
from the rest of the front and we can remove this unknot. Otherwise
we can use the first tangle replacement from Corollary
\ref{cor:half-clasp-augmentation} to eliminate the leftmost crossing on
this strand:

\begin{center}
\includegraphics{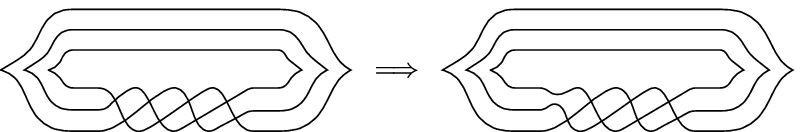}
\par\end{center}

\noindent We can repeat this procedure, removing crossings and
unlinked Legendrian unknots, until $K$ has become a disjoint union
of such unknots. Such a front always admits an augmentation, so by
Corollary \ref{cor:half-clasp-augmentation} we see that $K$ has
an augmentation. (The set of augmentations of $K$ was described by
K{\'a}lm{\'a}n \cite{Kalman:2006p630}, who also described an analogous 
{}``Seifert ruling'' of $K$.)
\end{example}

\subsection{Unhooking a clasp}

Let $X$ and $C$ be tangles consisting of a pair of interlocking
left and right cusps and a pair of disjoint left and right cusps,
respectively, and consider the effect of replacing $X$ with $C$
in a front:

\begin{center}
\includegraphics{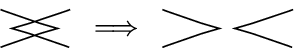}
\par\end{center}

We have already computed the DGA $D(\tilde{C})$ in the previous section:
it is generated freely by $p,q,$ and $\rho_{ij}$ with $\partial p=1+\rho_{12}$,
$\partial q=1+\rho_{34}$, and $|p|=|q|=1$. On the other hand, $\tilde{X}$
is constructed as follows:

\begin{center}
\includegraphics{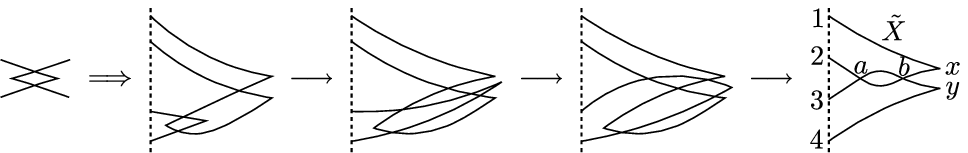}
\par\end{center}

\noindent and so $D(\tilde{X})$ is generated by $x,y,a,b,$ and $\rho_{ij}$
satisfying \begin{eqnarray*}
\partial x & = & 1+\rho_{12}(ab+1)+\rho_{13}b\\
\partial y & = & 1+(ba+1)\rho_{34}+b\rho_{24}\\
\partial a & = & \rho_{23}\\
\partial b & = & 0.\end{eqnarray*}
Let $D''$ be the DGA constructed by adding free generators $c$ and
$d$ to $D(\tilde{X})$ satisfying $\partial c=b$ and $\partial d=a+\rho_{13}+(x+(\rho_{12}a+\rho_{13})c)\rho_{23}$.
\begin{lem}
The DGA $D''$ is stable tame isomorphic to $D(\tilde{C})$.\end{lem}
\begin{proof}
We start by applying the sequence of tame isomorphisms\begin{eqnarray*}
x & \to & x + (\rho_{12}a + \rho_{13})c \\
y & \to & y + c(a\rho_{34} + \rho_{24}) \\
a & \to & a + \rho_{13} + x\rho_{23}.
\end{eqnarray*}
to $D''$; now $\partial x = 1+\rho_{12}$, $\partial y = 1+\rho_{34}$,
$\partial d = a$, and $\partial a = 0$.  Next, we destabilize twice to remove
the  pairs of generators $(d,a)$ and $(c,b)$, and relabel $x$ and $y$ by $p$
and $q$.  We now have the DGA generated by
$p$, $q$, and $\rho_{ij}$ with $\partial p = 1+\rho_{12}$ and
$\partial q = 1+\rho_{34}$, which is precisely $D(\tilde{C})$.
\end{proof}
\begin{prop}
\label{pro:unhook-clasp}If $K'$ is obtained from $K$ by replacing
the tangle $X$ with the tangle $C$, then $Ch(K)$ and $Ch(K')$ are stable
tame isomorphic to algebras $\mathcal{A}$ and $\mathcal{A'}$, where
$\mathcal{A'}$ is obtained from $\mathcal{A}$ by adding two free generators.
If $Ch(K')$ admits an augmentation, then so does $Ch(K)$.
\end{prop}
\begin{example}
\noindent \label{exa:whitehead-double}Given the Legendrian Whitehead
double $K_{\mathrm{dbl}}(k,l)$ of a front $K$ as defined by Fuchs
\cite{Fuchs:2003p674}, we can unhook the clasp and perform $k+l$
type I Reidemeister moves to remove the extra twists from the remaining
knot:

\noindent \begin{center}
\includegraphics{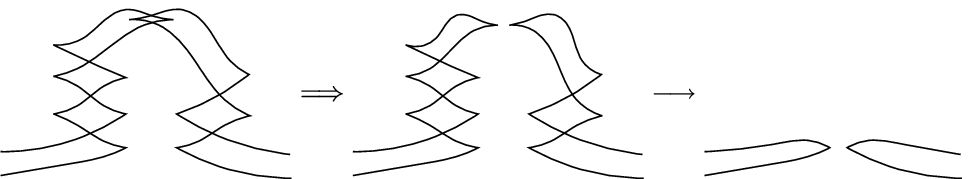}
\par\end{center}

\noindent The resulting knot is a Legendrian unknot, which admits
an augmentation, so by Proposition \ref{pro:unhook-clasp} we recover
Fuchs's result that $K_{\mathrm{dbl}}(k,l)$ does as well for all
$k,l\geq0$.\end{example}
\begin{prop}
\label{pro:whitehead-double-polynomial}Suppose that $K$ has rotation
number $r$. Then $W(K)=K_{\mathrm{dbl}}(0,0)$ admits an augmentation
with Chekanov polynomial $t+t^{2r}+t^{-2r}$.\end{prop}
\begin{proof}
Let $U$ denote the Legendrian unknot. It is easy to check that if
we unhook the clasp as in Example \ref{exa:whitehead-double}, 
then the DGA stable tame isomorphic to $Ch(U)$ is obtained from $Ch(W(K))$
by adding two extra generators 
$c$ and $d$ in degrees $1\pm2r$ since the corresponding crossings $a$ and $b$
have gradings
$\pm2r$. Any augmentation $\epsilon'$ of $Ch(U)$ gives an augmentation
$\epsilon$ of $Ch(W(K))$ by composition with the inclusion
$\iota: Ch(W(K)) \hookrightarrow Ch(U)$, and $\epsilon'$
must have Chekanov polynomial $P_{\epsilon'}(t)=t$. Furthermore,
$\iota$ induces an inclusion on the linearizations
$A^{W(K),\epsilon}\hookrightarrow A^{U,\epsilon'}$ whose cokernel
is the chain complex $C=\mathbb{F}(c+\epsilon(c))\oplus\mathbb{F}(d+\epsilon(d))$.
Note that the differential on $C$ must be trivial since $|c|-|d|$
is even, and so $H_{*}(C)\cong\mathbb{F}_{1-2r}\oplus\mathbb{F}_{1+2r}$
where the subscripts denote degrees.

The short exact sequence of chain complexes \[
0\to A^{W(K),\epsilon}\to A^{U,\epsilon'}\to C\to0\]
gives a long exact sequence in homology, so for example the sequence
\[
H_{i+1}(A^{U,\epsilon'})\to H_{i+1}(C)\to H_{i}(A^{W(K),\epsilon})\to H_{i}(A^{U,\epsilon'})\]
is exact, and thus when $i\not=0,1$ we have $H_{i}(A^{W(K),\epsilon})\cong H_{i+1}(C)$.
In particular, if $i\not\in\{0,1,\pm2r\}$ it follows that $H_{i}(A^{W(K),\epsilon})=0$;
and if $r\not=0$ then $H_{\pm2r}(A^{W(K),\epsilon})\cong H_{1\pm2r}(C)\cong\mathbb{F}$.
We also get an exact sequence \[
0\to H_{1}(A^{W(K),\epsilon})\to\mathbb{F}\to H_{1}(C)\to H_{0}(A^{W(K),\epsilon})\to0\]
since $H_{2}(C)\cong H_{0}(A^{U,\epsilon'})\cong0$ and $H_{1}(A^{U,\epsilon'})\cong\mathbb{F}$. 

By considering $W(K)$ as the closure of a long Legendrian knot, \cite[Theorem 12.4]{Chekanov:2002p539}
shows that the homology group $H_{1}(A^{W(K),\epsilon})$ must be
nontrivial. Thus the injection $H_{1}(A^{W(K),\epsilon})\to\mathbb{F}$
in the last exact sequence is an isomorphism, hence the map $H_{1}(C)\to H_{0}(A^{W(K),\epsilon})$
must be an isomorphism as well. But $H_{1}(C)$ is zero if $r\not=0$
and $\mathbb{F}^{2}$ otherwise, so this determines $H_{0}(A^{W(K),\epsilon})$
and our computation of $H_{*}(A^{W(K),\epsilon})$ is complete; in
particular, its Poincar\'{e} polynomial is $t+t^{2r}+t^{-2r}$, as desired.\end{proof}
\begin{prop}
Suppose that $r=r(K)$ is nonzero. Then every augmentation of $W(K)$
has Chekanov polynomial $t+t^{2r}+t^{-2r}$.\end{prop}
\begin{proof}
Let $W(K)$ be divided into left half $W^{A}$ and right half $\tilde{X}$.
Then the elements of $D(\tilde{X})$ have gradings $|x|=|y|=1$; $|a|,|b|=\pm2r$;
$|\rho_{12}|=|\rho_{34}|=0$; $|\rho_{13}|=|\rho_{24}|=|a|$; and
$|\rho_{23}|=|a|-1$ and $|\rho_{14}|=|a|+1$. Since $\rho_{12}$
and $\rho_{34}$ are the only generators in grading $0$, all others
must be in $\ker(\epsilon_{\tilde{X}})$ for any augmentation $\epsilon_{\tilde{X}}$;
and then from $\epsilon_{\tilde{X}}(\partial x)=\epsilon_{\tilde{X}}(\partial y)=0$
we get $\epsilon_{\tilde{X}}(\rho_{12})=\epsilon_{\tilde{X}}(\rho_{34})=1$. 

If we replace $\tilde{X}$ with $\tilde{C}$, so that we have a Legendrian
unknot $U$ divided into $W^{A}$ and $\tilde{C}$, the Maslov potential
of each strand remains unchanged, so $|\rho_{ij}|$ can still only
be nonzero for $\rho_{12}$ and $\rho_{34}$, and then $\epsilon_{\tilde{C}}(\partial p)=\epsilon_{\tilde{C}}(\partial q)=0$
forces $\epsilon_{\tilde{C}}(\rho_{12})=\epsilon_{\tilde{C}}(\rho_{34})=1$
as well. In particular, both $D(\tilde{X}$) and $D(\tilde{C})$ have
a unique augmentation, and these take the same values on the elements
$\rho_{ij}$, so an augmentation of $A(W^{A})$ extends to an augmentation
of $W(K)$ iff it extends to an augmentation of $U$. Thus every augmentation
$\epsilon$ of $Ch(W(K))$ is the pullback of one on $Ch(U)$: construct
$\epsilon':Ch(U)\to\mathbb{F}$ by setting $\epsilon'(v)=\epsilon(v)$
for every vertex $v$ of $W^{A}$ and $\epsilon'(v)=0$ on the vertices
of $\tilde{C}$, and then $\epsilon$ is exactly the composition $Ch(W(K))\hookrightarrow Ch(U)\stackrel{\epsilon'}{\to}\mathbb{F}$.
But we showed in the proof of Proposition \ref{pro:whitehead-double-polynomial}
that such an augmentation must have Chekanov polynomial $t+t^{2r}+t^{-2r}$,
and so $W(K)$ cannot have any other Chekanov polynomials.
\end{proof}
On the other hand, when $r(K)=0$, we can ask the following:
\begin{question}
\label{con:augs-of-double}Suppose $K$ is a Legendrian knot with $r(K)=0$.
Does the Whitehead double of $K$ have Chekanov polynomials other than $t+2$?
\end{question}
In particular, this has been checked using a program written in 
Sage \cite{Stein:sage} for all but two of the fronts
in Melvin and Shrestha's table \cite{Melvin:2005p636}, which includes
one $tb$-maximizing front for each knot up through 9 crossings and
their mirrors. The answer is yes for every front which admits an augmentation
except the Legendrian unknot, and no for every front which does not except
for $m(9_{42})$.  (The unknown cases are $m(8_{5})$ and $m(9_{30})$,
neither of which admits an augmentation.)

In the case of $m(9_{42})$, whose extra Chekanov polynomial is $t^{2}+2t+2+t^{-1}+t^{-2}$,
we note that the Kauffman bound on $tb$ is not tight; equivalently, this 
knot does not admit an ungraded augmentation \cite{Rutherford}.
This is the only
such knot up to nine crossings for which a $tb$-maximizing representative
has $r=0$ \cite{Ng:MR1852765}, so
we speculate that these phenomena are related. (The other knot which
does not achieve the Kauffman bound is the $(4,-3)$ torus knot $m(8_{19})$,
for which $\overline{tb}=-12$ and so $r$ must be odd.) On the other
hand, the $m(10_{132})$ representative with $tb=-1$ and $r=0$ in
\cite[Figure 7]{Ng:MR2186113} has no Chekanov polynomials other than $t+2$
even though it does not admit an ungraded augmentation.

\section{Augmentations of Whitehead doubles}

In this section we prove the following result, which answers Question
\ref{con:augs-of-double} for Legendrian knots $K$ with augmentations
satisfying $P_{\epsilon}(t)\not=t$:
\begin{thm}
\label{thm:extra-double-aug}Let $K$ be a front with rotation number
0, and suppose that $K$ has an augmentation $\epsilon$ with Chekanov
polynomial $P_{\epsilon}(t)=t+\sum a_{i}t^{i}$. Then its Legendrian
Whitehead double $W(K)$ has an augmentation $\epsilon'$ with $P_{\epsilon'}(t)=t+2+(t+2+t^{-1})\sum a_{i}t^{i}$.
\end{thm}
As a sample application, we have a new proof of the following result
of Melvin and Shrestha \cite{Melvin:2005p636}:
\begin{cor}
There are prime Legendrian knots with arbitrarily many Chekanov polynomials.\end{cor}
\begin{proof}
Let $K_{0}$ be any Legendrian knot with rotation number $0$, and
for $n\geq1$ let $K_{n}$ be the Legendrian Whitehead double of $K_{n-1}$;
then each $K_{n}$ is prime because Whitehead doubles have genus $1$.
We claim that $K_{n}$ has at least $n$ distinct Chekanov polynomials.

If we define a sequence of Laurent polynomials $p_{1}(t)=t+2$, $p_{2}(t)=3t+6+2t^{-1}$,
and so on by the formula \[
p_{n}(t)=t+2+(t+2+t^{-1})(p_{n-1}(t)-t),\]
then we can explicitly solve for $p_{n}(t)$ as \[
p_{n}(t)=\frac{2(t+2+t^{-1})^{n}+t^{2}+t-1}{t+1+t^{-1}}=t+2\sum_{k=1}^{n}{n \choose k}(t+1+t^{-1})^{k-1}\]
and so the $p_{n}$ are all distinct. But for any $n\geq1$, the polynomials
$p_{1}(t),\dots,p_{n}(t)$ are all Chekanov polynomials of $K_{n}$:
$p_{1}$ is for all $n$ by Proposition \ref{pro:whitehead-double-polynomial},
and if $p_{1},\dots,p_{i-1}$ are Chekanov polynomials of $K_{i-1}$
then Theorem \ref{thm:extra-double-aug} guarantees that $p_{2},\dots,p_{i}$
are Chekanov polynomials of $K_{i}$, so the claim follows by induction.\end{proof}
\begin{rem}
Since $p_{n}(1)=\frac{1}{3}(2\cdot4^{n}+1)$, the ranks of the corresponding
linearized homologies are all distinct as well.

If instead we take $r(K_{0})\not=0$ then $p_{1},\dots,p_{n-1}$ are
Chekanov polynomials of $K_{n}$ by applying this argument to $K'_{0}=W(K_{0})=K_{1}$
and $K'_{n-1}=K_{n}$, and since $K_{1}$ has Chekanov polynomial
$t+t^{2r}+t^{-2r}$ we get an $n$th Chekanov polynomial of degree
$2r+n-1$ and rank $\frac{1}{3}(2\cdot4^{n}+1)$ for $K_{n}$ by applying
the same recurrence to $t+t^{2r}+t^{-2r}$ a total of $n-1$ times.
Thus any $n$-fold iterated Legendrian Whitehead double has at least
$n$ distinct Chekanov polynomials.
\end{rem}
The Whitehead double differs from the $2$-copy, defined in
\cite{Mishachev:2003p676},
by a single crossing, or more precisely by replacing the tangle $\tilde{X}$
from the previous section with the tangle $\tilde{P}$:

\begin{center}
\includegraphics{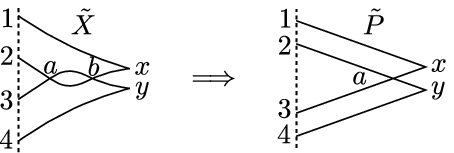}
\par\end{center}

\noindent Thus we will start by analyzing the closely related $2$-copy
$C(K)$, where we have fixed gradings so that for any two parallel
strands of $C(K)$, the top strand has Maslov potential one greater
than that of the bottom strand. (In particular, since $r(K)=0$ we
have $|a|=|b|=0$ in $\tilde{X}$, so we may assume without loss of
generality that the potentials of the strands points $1,2,3,4$ on
the dividing line are $1,0,0,-1$ in both tangles.) We will split
a linearization of $C(K)$ into four subcomplexes, compute the homology
of three and a half of these, and use this information to recover
the linearized homology of $W(K)$.

It follows from Corollary \ref{cor:half-clasp-augmentation} that
$Ch(C(K))$ is stable tame isomorphic to a DGA obtained by adding
a free generator $g$ to $Ch(W(K))$ in grading $1$. One can check
that it suffices to let $\partial g=b+1$, but we do not need this
fact.

We will assume for the rest of this section that $K$ is a fixed Legendrian
knot with $r(K)=0$ and augmentation $\epsilon:Ch(K)\to\mathbb{F}$,
and that we have a simple front for $K$.

\subsection{Constructing augmentations of $C(K)$ and $W(K)$}

For each crossing $c$ of $K$, the 2-copy $C(K)$ has four corresponding
crossings, which we will label $c_{N}$, $c_{E}$, $c_{S}$, and $c_{W}$
in clockwise order from the top. It is easy to check that the gradings
of these crossings are $|c|$, $|c|+1$, $|c|$, and $|c|-1$ respectively.
Let $K_{1},K_{2}\subset C(K)$ denote the upper and lower (in the
$z$-direction) copies of
$K$, respectively, so that both strands through $c_{N}$ belong to
$K_{1}$, both strands through $c_{S}$ belong to $K_{2}$, and $c_{E}$
and $c_{W}$ involve strands from both $K_{1}$ and $K_{2}$.

\begin{figure}
\begin{centering}
\includegraphics{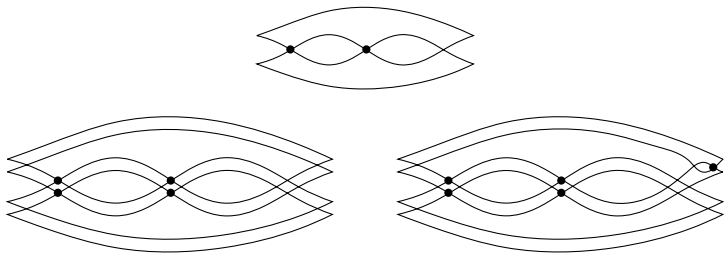}
\par\end{centering}

\caption{An augmentation $\epsilon$ of $K$ and the corresponding augmentations
$\epsilon'$ of the 2-copy and Whitehead double. The vertices $v$
with $\epsilon(v)=1$ are indicated by black dots.\label{fig:double-augs}}

\end{figure}

\begin{prop}
Define an algebra homomorphism $\epsilon':Ch(C(K))\to\mathbb{F}$
as in Figure \ref{fig:double-augs} by $\epsilon'(c_{N})=\epsilon'(c_{S})=1$
whenever $\epsilon(c)=1$, and $\epsilon'(v)=0$ for all other vertices
of $v$. Then $\epsilon'$ is an augmentation of $Ch(C(K))$.\end{prop}
\begin{proof}
Since $\epsilon'$ is a ``proper'' augmentation in the sense of 
\cite{Mishachev:2003p676}, meaning that $\epsilon'(v)=0$ whenever the
strands through $v$ belong to different components of $C(K)$, this is a
special case of \cite[Proposition 3.3c]{Mishachev:2003p676}.  More generally,
any augmentations $\epsilon_1$ and $\epsilon_2$ of $K_1$ and $K_2$ uniquely
determine a proper augmentation of $C(K)$.
\end{proof}
Let $\mathcal{A} \cong Ch(C(K))$ be the DGA which is 
is constructed by adding a generator $g$ to $Ch(W(K))$ as in Corollary
\ref{cor:half-clasp-augmentation}.  The
inclusion $Ch(W(K))\hookrightarrow\mathcal{A}$ induces an augmentation
of $Ch(W(K))$ which we will also call $\epsilon'$; one can show
that it satisfies $\epsilon'(a)=0$ and $\epsilon'(b)=1$, and is
defined identically to the augmentation of $Ch(C(K))$ on all other
vertices.

The inclusion $Ch(W(K))\hookrightarrow\mathcal{A}$ induces a map
on the linearized complexes which we can extend to a short exact sequence
as before, \[
0\to A^{W(K),\epsilon'}\to A^{C(K),\epsilon'}\to\mathbb{F}_{1}\to0,\]
where we are using $A^{C(K),\epsilon'}$ to mean the linearization of
$\mathcal{A}$ since
they have the same homology, and the cokernel $\mathbb{F}_{1}$ is
generated in degree $1$ by $g$. The corresponding long exact sequence
in homology tells us that $H_{i}(A^{W(K),\epsilon'})\cong H_{i}(A^{C(K),\epsilon'})$
for $i\not=0,1$, and in particular for all $i<0$.

\subsection{The linearized homology of $C(K)$}

Mishachev showed in \cite{Mishachev:2003p676} that the DGA of
$C(K)$ splits as $Ch(C(K))=\bigoplus_{i\in\mathbb{Z}}\mathcal{A}_{i}$,
where a vertex $v$ is in $\mathcal{A}_{-1}$ if the top and bottom
strands through $v$ are in $K_{2}$ and $K_{1}$ respectively, in
$\mathcal{A}_{1}$ if the top and bottom strands are in $K_{1}$ and
$K_{2}$ respectively, and in $\mathcal{A}_{0}$ if both strands through
$v$ belong to the same component; and if $v\in\mathcal{A}_{i}$ and
$v'\in\mathcal{A}_{i'}$, then $vv'\in\mathcal{A}_{i+i'}$. This splitting
extends to the linearization with respect to the augmentation $\epsilon'$,
but in fact we can split the linearized complex even more:
\begin{prop}[\cite{Ng:2003p540}]
\label{pro:splitting-2-copy}There is a splitting \[
A^{C(K),\epsilon'}\cong\bigoplus_{d\in\{N,E,S,W\}}A_{d}^{C(K),\epsilon'}\]
where the N, E, S, and W subcomplexes are generated by the vertices
whose top and bottom strands belong to components $(K_{1},K_{1})$,
$(K_{1},K_{2})$, $(K_{2},K_{2})$, and $(K_{2},K_{1})$, respectively.\end{prop}

The cusps and crossings of $K$ determine several types of vertices of $C(K)$:
\begin{center}
\includegraphics{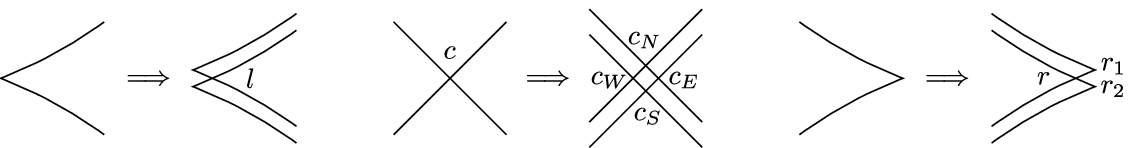}
\par\end{center}
In this picture, the crossings $l$, $r$, $r_1$, and $r_2$ belong to the E, W,
N, and S subcomplexes, respectively, and each $c_{d}$ belongs to
$A^{C(K),\epsilon'}_{d}$.

\begin{lem}
There is an isomorphism $H_{i}(A_{W}^{C(K),\epsilon'})\cong H_{i+1}(A^{K,\epsilon})$
for all $i\in\mathbb{Z}$.\end{lem}
\begin{proof}
The subcomplex $A_{W}^{C(K),\epsilon'}$ is generated by crossings $c_{W}$
corresponding to crossings $c$ of $K$, as well as crossings $r$ adjacent
to pairs of right cusps, so the generators of the $i$th graded component of
$A_{W}^{C(K),\epsilon'}$ are in bijection with the generators of
$(A^{K,\epsilon})_{i+1}$.  Let $v_{W}$ and $v$ denote a generator of
$A_{W}^{C(K),\epsilon'}$ and the corresponding generator of $A^{K,\epsilon}$,
respectively.

No disk $D$ contributing to $\partial^{\epsilon'}v_{W}$
can have more than one unaugmented corner, so in particular as we travel along
$\partial D$ we cannot switch between components of $C(K)$ more than once.
Since $\partial D$ leaves $v_{W}$ along $K_{2}$
and returns along $K_{1}$ when traveling counterclockwise, it must switch at
some crossing $c'_{W}$, which is
then the unique unaugmented corner of $D$ and so $D$ contributes $c'_{W}$ to
$\partial^{\epsilon'}v_{W}$.

Now consider the linearized differential $\partial^{\epsilon}v\in A^{K,\epsilon}$.
Each disk $D'$ with initial vertex $v$ and a single unaugmented
corner $c'$ corresponds to a unique disk $D$ for $v_{W}$ with corner
$c'_{W}$ as described above. On the other hand, if every corner $c'_{j}$
of $D'$ is augmented then $D'$ contributes $\sum c'_{j}$ to
$\partial^{\epsilon}v$.
In this case $D'$ corresponds to one disk $D_{j}$ which contributes to
$\partial^{\epsilon'} v_{W}$
for each corner $c'_{j}$: this is the disk $D_{j}$ with augmented
corners $(c'_{k})_{S}$ for all $k<j$, then an unaugmented corner
at $(c'_{j})_{W}$, and then augmented corners $(c'_{k})_{N}$ for all
$k>j$, hence $D_{j}$ contributes $(c'_{j})_{W}$ to $\partial^{\epsilon'}v_{W}$
for each $j$ and the total contribution is $\sum(c'_{j})_{W}$.  Finally, if
$v$ is a right cusp then $\partial v$ contains an extra $1$ which does not
appear in $\partial v_{W}$, but this does not contribute to the linearization
$\partial^{\epsilon} v$.

We conclude that if $\partial^{\epsilon}v=\sum c'_{j}$ then
$\partial^{\epsilon'}v_{W}=\sum(c'_{j})_{W}$
for all $v$, and the desired isomorphism follows immediately.\end{proof}
\begin{lem}
Both $H_{*}(A_{N}^{C(K),\epsilon'})$ and $H_{*}(A_{S}^{C(K),\epsilon'})$
are isomorphic to $H_{*}(A^{K,\epsilon})$.\end{lem}
\begin{proof}
See \cite[Section 2.5]{Ng:2003p540}, in particular the discussion after
Definition 2.20.
\end{proof}
The only remaining subcomplex is $A_{E}^{C(K),\epsilon'}$. This complex
is more complicated than the others, but it is still accessible in
negative degree:
\begin{lem}
There is an isomorphism $H_{i}(A_{E}^{C(K),\epsilon'})\cong H_{i-1}(A^{K,\epsilon})$
for all $i<0$.\end{lem}
\begin{proof}
The complex $A_{E}^{C(K),\epsilon'}$ is generated by crossings $c_{E}$,
with $|c_{E}|=|c|+1$, as well as the crossings $l_{i}$ in between
each pair of left cusps, satisfying $|l_{i}|=0$ and $\partial l_{i}=0$.
For each $i<0$, we have an isomorphism of graded components $(A_{E}^{C(K),\epsilon'})_{i}\cong(A^{K,\epsilon})_{i-1}$
matching each $c_{E}$ to $c$, since the complexes differ only by
the generators $l_{i}$ in grading $0$ and the right cusps of $K$
in grading $1$. The differentials are identical under this identification
just as before, except we do not have to consider disks in $K$ with
all corners augmented since this can only happen for $v\in(A^{K,\epsilon})_{1}$.
Furthermore, the image of $\partial^{\epsilon'}:(A_{E}^{C(K),\epsilon'})_{0}\to(A_{E}^{C(K),\epsilon'})_{-1}$
is identical to that of $\partial^{\epsilon}:(A^{K,\epsilon})_{-1}\to(A^{K,\epsilon})_{-2}$
since the extra generators $l_{i}$ do not contribute to $\mathrm{im}(\partial^{\epsilon'})$.
Thus we have an isomorphism $H_{i}(A_{E}^{C(K),\epsilon'})\cong H_{i-1}(A^{K,\epsilon})$
for all $i<0$.
\end{proof}

\begin{proof}[Proof of Theorem \ref{thm:extra-double-aug}]
We have now computed $H_{i}(A^{W(K),\epsilon'})$ for all $i<0$:
namely, it is isomorphic to $H_{i}(A^{C(K),\epsilon'})$, and then
the splitting of $A^{C(K),\epsilon'}$ gives an isomorphism \[
H_{i}(A^{W(K),\epsilon'})\cong H_{i+1}(A^{K,\epsilon})\oplus(H_{i}(A^{K,\epsilon}))^{\oplus2}\oplus H_{i-1}(A^{K,\epsilon}).\]
If $P_{\epsilon}^{K}(t)$ and $P_{\epsilon'}^{W(K)}(t)$ are the Chekanov
polynomials of $\epsilon$ and $\epsilon'$, then, it follows that
$P_{\epsilon'}^{W(K)}(t)=(t+2+t^{-1})P_{\epsilon}^{K}(t)+f(t)$ for
some actual polynomial $f\in\mathbb{Z}[t]$, since the coefficient
of $t^{i}$ on either side is the rank of the corresponding $i$th
homology group for $i<0$. By Poincar\'{e} duality \cite{Sabloff:2006p526}
we can write $P_{\epsilon}^{K}(t)=t+\sum a_{i}t^{i}$ and $P_{\epsilon'}^{W(K)}(t)=t+\sum b_{i}t^{i}$,
where $a_{i}=a_{-i}$ and $b_{i}=b_{-i}$ for all $i$; then \[
t+\sum b_{i}t^{i}=(t^{2}+2t+1)+\sum(a_{i+1}+2a_{i}+a_{i-1})t^{i}+f(t)\]
or \[
t^{2}+t+1+f(t)=\sum(b_{i}-a_{i+1}-2a_{i}-a_{i-1})t^{i}.\]
The coefficients $c_{i}$ on the right hand side are symmetric, i.e.
they satisfy $c_{i}=c_{-i}$, so the left hand side must be symmetric
as well, and since it is a polynomial rather than a Laurent series
we must have $f(t)=n-t^{2}-t$ for some $n\in\mathbb{Z}$. Therefore
\[
P_{\epsilon'}^{W(K)}(t)=t+(n+1)+(t+2+t^{-1})\sum a_{i}t^{i}.\]
In order to determine $n$, we note that $P_{\epsilon'}^{W(K)}(-1)=tb(W(K))=1$,
and substituting $t=-1$ into the above equation leaves $n=1$. We
conclude that \[
P_{\epsilon'}^{W(K)}(t)=t+2+(t+2+t^{-1})\sum a_{i}t^{i},\]
as desired.
\end{proof}

\section{The characteristic algebra}

\subsection{The van Kampen theorem for the characteristic algebra}

Ng \cite{Ng:2003p540} introduced the characteristic algebra of a
Legendrian knot as an effective way to distinguish knots using the
Chekanov-Eliashberg algebra when the Chekanov polynomials could not.
\begin{defn}
Let $A$ be a DGA, and let $I\subset A$ be the two-sided ideal generated
by the image of $\partial$. The \emph{characteristic algebra} of
$A$ is the quotient $\mathcal{C}(A)=A/I$, with grading inherited
from $A$.\end{defn}
Two characteristic algebras $A_1/I_1$ and $A_2/I_2$ are \emph{stable
tame isomorphic} if we can add some free generators to one or both algebras
to make them tamely isomorphic.

\begin{thm}[{\cite[Theorem 3.4]{Ng:2003p540}}]
 The stable tame isomorphism class of the characteristic algebra
$\mathcal{C}(Ch(K))$ of a Legendrian knot is a Legendrian isotopy
invariant.
\end{thm}
There are some technicalities involved in defining equivalence
-- in particular, one must consider equivalence relations on the pair
$(A,I)$ rather than the quotient $A/I$ -- but we will ignore these
since we are only concerned with the stable isomorphism class of $\mathcal{C}(Ch(K))$.
For example, Ng showed that the isomorphism class (together with the
gradings of the generators of $Ch(K)$) is strong enough to recover
the first and second order Chekanov polynomials of $K$, and he conjectured
that the Chekanov polynomials of all orders are determined by this
information.

We can define $\mathcal{C}$ as a functor from the category of DGAs
to the category of graded associative unital algebras: we have already
defined it for objects of the category, and given a DGA morphism $f:X\to Y$
we note that the relation $\partial f=f\partial$ implies $f(\partial(X))=\partial(f(X))\subset\partial(Y)$,
hence $f$ descends to a morphism $\mathcal{C}(f):\mathcal{C}(X)\to\mathcal{C}(Y)$.
It turns out that this functor is well-behaved.
\begin{prop}
The functor $\mathcal{C}$ preserves pushouts.\end{prop}
\begin{proof}
It suffices to prove that the functor $\mathcal{D}:\mathrm{GA}\to\mathrm{DGA}$
(here $\mathrm{GA}$ denotes graded algebras) defined by $\mathcal{D}(X)=(X,\partial_{X}=0)$
is a right adjoint to $\mathcal{C}$; then, since $\mathcal{C}$ is
a left adjoint it preserves colimits, which include pushouts.

Given a DGA $(A,\partial)$ and a graded algebra $X$, we need to
establish a natural bijection \[
\varphi:\hom_{\mathrm{GA}}(\mathcal{C}(A),X)\to\hom_{\mathrm{DGA}}(A,\mathcal{D}(X)).\]
Letting $\pi:A\to\mathcal{C}(A)$ denote the projection of graded
algebras, we can define $\varphi(f)=f\circ\pi:A\to X$ for any $f\in\hom_{\mathrm{GA}}(\mathcal{C}(A),X)$.
This is in fact a chain map since $(f\circ\pi)\circ\partial_{A}=f\circ(\pi\circ\partial_{A})=0=\partial_{X}\circ(f\circ\pi)$,
so $\varphi(f)\in\hom_{\mathrm{DGA}}(A,\mathcal{D}(X))$, and it is
clear that $\varphi$ is injective. Conversely, given a chain map
$\tilde{g}\in\hom_{\mathrm{DGA}}(A,\mathcal{D}(X))$ we must have
$\tilde{g}(\partial_{A}a)=\partial_{X}(\tilde{g}(a))=0$, and since
$\tilde{g}$ vanishes on the image of $\partial_{A}$ it factors through
the graded algebra $\mathcal{C}(A)$, hence $\tilde{g}=\varphi(g)$
for some $g\in\hom_{\mathrm{GA}}(\mathcal{C}(A),X)$ and so $\varphi$
is surjective. Since $\varphi$ is also clearly natural, we conclude
that $\mathcal{C}$ and $\mathcal{D}$ are adjoints, as desired.
\end{proof}
The following version of van Kampen's theorem for characteristic algebras
is now an immediate consequence of Theorems \ref{thm:pushout-square}
and \ref{thm:generalized-pairing}.
\begin{thm}
\label{thm:van-kampen-characteristic}Let $K$ be a simple Legendrian
front split by a vertical dividing line into a left half $K^{A}$
and a right half $K^{D}$; or, let $K_{1}$, $K_{2}$, and $K_{3}$
be adjacent regions of a simple front with $K_{12}=K_{1}\cup K_{2}$,
$K_{23}=K_{2}\cup K_{3}$, and $K_{123}=K_{1}\cup K_{2}\cup K_{3}$.
Then the diagrams \begin{eqnarray*}
\xymatrix{\mathcal{C}(I_{n})\ar[r]\ar[d]_{\mathcal{C}(w)} & \mathcal{C}(D(K^{D}))\ar[d]^{\mathcal{C}(w')}\\
\mathcal{C}(A(K^{A}))\ar[r] & \mathcal{C}(Ch(K))}
 &  & \xymatrix{\mathcal{C}(DA(K_{2}))\ar[r]\ar[d]_{\mathcal{C}(w)} & \mathcal{C}(DA(K_{23}))\ar[d]^{\mathcal{C}(w')}\\
\mathcal{C}(DA(K_{12}))\ar[r] & \mathcal{C}(DA(K_{123}))}
\end{eqnarray*}
are pushout squares in the category of graded algebras.\end{thm}
\begin{rem}
Although the DGA morphism $I_{n}\to D(K^{D})$ is an inclusion, this
is generally not true of the induced map $\varphi:\mathcal{C}(I_{n})\to\mathcal{C}(D(K^{D}))$.
Suppose that $K^{D}$ has some crossings but no left cusps, and let
$v$ be a leftmost crossing of $K^{D}$. If strands $s_{1}$ and $s_{2}$
pass through $v$, then $s_{1}$ and $s_{2}$ cannot intersect any
other strands between the dividing line and $v$, so we must have
$\partial v=\rho_{i,i+1}$ for some $i$. But now $\varphi(\rho_{i,i+1})=0$,
and yet $\rho_{i,i+1}\in\mathcal{C}(I_{n})$ cannot be zero since
the two-sided ideal $\mathrm{Im}(\partial)\subset I_{n}$ is generated
by homogeneous quadratic terms.
\end{rem}
Let $\mathcal{C}'$ denote the composition of $\mathcal{C}$ with
the abelianization functor from graded algebras to graded commutative
algebras. Since abelianization also preserves pushouts, the abelianized
characteristic algebra $\mathcal{C}'(Ch(K))$ satisfies Theorem \ref{thm:van-kampen-characteristic}
as well; in this category, pushouts are tensor products, so for example
we can express this as \[
\mathcal{C}'(DA(K_{123}))\cong\mathcal{C}'(DA(K_{12}))\otimes_{\mathcal{C}'(DA(K_{2}))}\mathcal{C}'(DA(K_{23})).\]

\subsection{Tangle replacement and the characteristic algebra}

The following is Conjecture 3.14 of \cite{Ng:2003p540}.
\begin{conjecture}
\label{con:characteristic-maximal-tb}Let $\mathcal{K}$ be any Legendrian
representative of the knot $K$ with maximal Thurston-Bennequin number.
Then the equivalence class of the ungraded abelianized characteristic
algebra $\mathcal{C}'(Ch(\mathcal{K}))$ is a topological invariant
of $K$.
\end{conjecture}
It is currently unknown whether there is a set of moves relating any
pair of topologically equivalent Legendrian links $L_{1}$ and $L_{2}$
with the same $tb$ \cite{Etnyre:2003p790}. A positive answer could
provide a straightforward way to resolve Conjecture \ref{con:characteristic-maximal-tb}:
\begin{prop}
\label{pro:tangle-d-diagram-2}Let $T_{1}$ and $T_{2}$ be Legendrian
tangles with $m$ strands on the left and $n$ strands on the right,
and let $\tilde{T}_{1}$ and $\tilde{T}_{2}$ be constructed as in
Figure \ref{fig:tangle-d-diagram}, where to a tangle $T$ we associate
the following half-diagram:

\begin{center}
\includegraphics{tangle-d-diagram}
\par\end{center}

\noindent If there is a stable isomorphism $\varphi:\mathcal{C}'(D(\tilde{T}_{1}))\to\mathcal{C}'(D(\tilde{T}_{2}))$
such that $\varphi(\rho_{ij})=\rho_{ij}$ for all $i$ and $j$, then
replacing $T_{1}$ with $T_{2}$ (or vice versa) in a front $K$ preserves
the stable isomorphism type of $\mathcal{C}'(Ch(K))$.\end{prop}
\begin{proof}
Repeat the proof of Proposition \ref{pro:tangle-d-diagram}, noting
that now we are working with pushouts of commutative algebras rather
than DGAs. The resulting commutative diagram \[
\xymatrix{\mathcal{C}'(D(\tilde{T}_{1}))\ar[r]^{\varphi}\ar[d] & \mathcal{C}'(D(\tilde{T}_{2}))\ar[d]\\
\mathcal{C}'(Ch(K_{1}))\ar[r]^{\tilde{\varphi}} & \mathcal{C}'(Ch(K_{2}))}
\]
is a pushout square, and since $\varphi$ is a stable isomorphism,
$\tilde{\varphi}$ must be as well.\end{proof}
\begin{rem}
Even when the map $\varphi$ isn't an isomorphism, we can still get
an interesting map relating $\mathcal{C}'(Ch(K_{1}))$ and $\mathcal{C}'(Ch(K_{2}))$.
For example, recall that the algebra $D(\tilde{C})$ of Theorem \ref{thm:parallel-break}
was obtained by adding an extra free generator $c$ to the algebra
$D(\tilde{P})$ with $\partial c=1+\rho_{12}$. Thus if $K'$ is obtained
from $K$ by the tangle replacement 

\begin{center}
\includegraphics{parallel-tangles}
\par\end{center}

\noindent it is easy to see that $\mathcal{C}'(Ch(K'))\cong(\mathcal{C}'(Ch(K))\otimes_{\mathbb{F}}\mathbb{F}[c])/\langle1+w(\rho_{12})\rangle$,
hence $\mathcal{C}'(Ch(K'))$ is stably isomorphic to a quotient of
$\mathcal{C}'(Ch(K))$.
\end{rem}
From now on we will abuse notation and write $\mathcal{C}'(K)=\mathcal{C}'(Ch(K))$,
$\mathcal{C}'(\tilde{T})=\mathcal{C}'(D(\tilde{T}))$, and so on whenever
it is clear from the description of the (partial) front whether we
are using the whole Chekanov-Eliashberg algebra or a type A, DA, or
D algebra.

\subsection{S and Z tangles}

\begin{figure}
\begin{centering}
\includegraphics{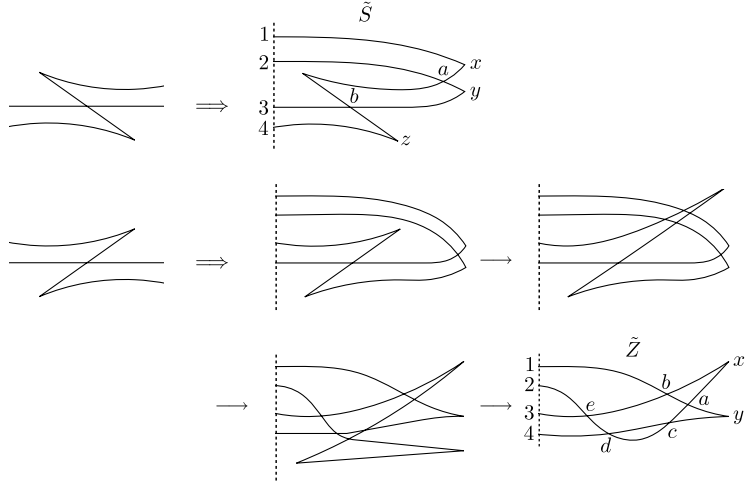}
\par\end{centering}

\caption{A pair of tangles which result in the same topological knot and value
of $tb$ but which may not preserve the Legendrian isotopy type. On
the right we show the half diagrams $\tilde{S}$ and $\tilde{Z}$
of Proposition \ref{pro:tangle-d-diagram-2}; note that for one we
perform several Legendrian Reidemeister moves to make it simple and
then eliminate some vertices.\label{fig:sz-move}}

\end{figure}

Figure \ref{fig:sz-move} gives an example of two tangles, an `S'
tangle and a `Z' tangle, which can be exchanged while preserving $tb$
and the topological knot type \cite{Etnyre:2003p790}.
\begin{thm}
\label{thm:S-equals-Z}Replacing an S tangle in a front diagram $K$
with a Z tangle, and vice versa, preserves the abelianized characteristic
algebra $\mathcal{C}'(K)$.\end{thm}
\begin{proof}
By Proposition \ref{pro:tangle-d-diagram-2} we only need to check
that $\mathcal{C}'(\tilde{S})$ and $\mathcal{C}'(\tilde{Z})$ are
stably isomorphic for the half diagrams $\tilde{S}$ and $\tilde{Z}$
of Figure \ref{fig:sz-move}.

The algebra $\mathcal{C}(\tilde{S})$ is generated by $x,y,z,a,b$
and $\rho_{ij}$, modulo the elements $\partial\rho_{ij}$ and \begin{eqnarray*}
\partial x & = & 1+\rho_{12}a\\
\partial y & = & 1+\rho_{23}+ab\\
\partial z & = & 1+b\rho_{34}\end{eqnarray*}
since $\partial a=\partial b=0$. Using $\partial\rho_{13}=\rho_{12}\rho_{23}=0$,
we see that \[
0=\rho_{12}+\rho_{12}\rho_{23}+\rho_{12}ab=\rho_{12}+b\]
and so $b=\rho_{12}$. Then from $\partial x=0$ we get $ba=1$, and
$\partial z=0$ implies $a=ab\rho_{34}=(1+\rho_{23})\rho_{34}=\rho_{34}$.
In particular, $\rho_{12}\rho_{34}=1$ and the generators $a$ and
$b$ are redundant; and then $\rho_{23}=1+ab=ba+ab$, which is zero
in the abelianization $\mathcal{C}'(\tilde{S})$. It follows that
\[
\mathcal{C}'(\tilde{S})=\mathbb{F}[x,y,z]\otimes_{\mathbb{F}}\left(\mathcal{C}'(I_{4})/\langle\rho_{12}\rho_{34}=1\rangle\right).\]

The algebra $\mathcal{C}'(\tilde{Z})$ is generated by $x,y,a,b,c,d,e$
and $\rho_{ij}$, modulo the elements $\partial\rho_{ij}$ and \begin{eqnarray*}
\partial x & = & 1+\rho_{34}c\\
\partial y & = & 1+\rho_{12}d+\rho_{14}+b\rho_{34}\\
\partial a & = & \rho_{12}(1+dc)+\rho_{14}c+b\rho_{34}c\\
\partial b & = & \rho_{12}e+\rho_{13}\\
\partial d & = & \rho_{24}+e\rho_{34}\\
\partial e & = & \rho_{23}\end{eqnarray*}
with $\partial c=0$. Applying $\rho_{34}c=1$ to the relation $\partial a=0$
yields \[
b=\rho_{12}(1+dc)+\rho_{14}c,\]
so the generator $b$ is redundant. Since $\partial a=\rho_{12}+c+(\partial y)c$,
we also have $c=\rho_{12}$, so once again $\rho_{12}\rho_{34}=1$
and then $\rho_{23}=0$ as before. Now $\partial b=0$ implies $e=\rho_{12}\rho_{34}e=\rho_{13}\rho_{34}$
and likewise $\partial d=0$ implies $e=\rho_{12}\rho_{24}$ (which
is the same element since $\partial\rho_{14}=0$), so $e$ is redundant
as well. We can conclude that \[
\mathcal{C}'(\tilde{Z})=\mathbb{F}[a,d,x,y]\otimes_{\mathbb{F}}\left(\mathcal{C}'(I_{4})/\langle\rho_{12}\rho_{34}=1\rangle\right),\]
and this is stably isomorphic to $\mathcal{C}'(\tilde{S})$, as desired.
\end{proof}
Legendrian twist knots have been classified by work of Etnyre, Ng,
and V{\'e}rtesi \cite{Etnyre:2010p894} which allows us to verify Conjecture
\ref{con:characteristic-maximal-tb} in this case. They prove that
any Legendrian representative of $K_{m}$ with non-maximal $tb$ can
be destabilized, so its characteristic algebra vanishes; up to orientation,
there is a unique representative maximizing $tb$ if $m\geq-1$ (where
$m=-1$ is the unknot); and for $m\leq-2$, any representative which
maximizes $tb$ can be isotoped to a front as in Figure \ref{fig:twist-knot},
where the rectangle is filled with $|m+2|$ negative half-twists,
each an S tangle or a Z tangle. Since we can replace any Z tangle
with an S tangle without changing $\mathcal{C}'(K)$, we can conclude:

\begin{figure}
\begin{centering}
\includegraphics{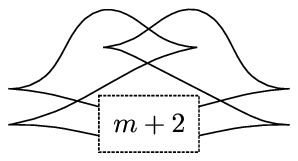}
\par\end{centering}

\caption{A front for the Legendrian twist knots $K_{m}$, $m\leq-2$.\label{fig:twist-knot}}

\end{figure}

\begin{cor}
Let $\mathcal{K}$ be a Legendrian representative of the twist knot
$K_{m}$. Then $\mathcal{C}'(\mathcal{K})$ depends only on $tb(\mathcal{K})$
and $m$.
\end{cor}
\begin{figure}
\begin{centering}
\includegraphics{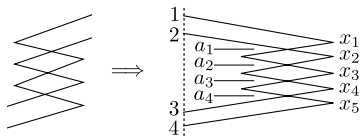}
\par\end{centering}

\caption{The $3S$ tangle and the associated partial front $\widetilde{3S}$.\label{fig:s-cubed}}

\end{figure}
In fact, many of these have the same abelianized characteristic algebra.
Consider the tangle $3S$ in Figure \ref{fig:s-cubed} obtained by
concatenating three S tangles. The crossings $a_{i}$ of $\widetilde{3S}$
have zero differential, whereas $\partial x_{i}$ takes the values
$1+\rho_{12}a_{1}$, $1+a_{1}a_{2}$, $1+a_{2}a_{3}$, $1+a_{3}a_{4}$,
and $1+a_{4}\rho_{34}$ for $1\leq i\leq5$, so $\mathcal{C}'(\widetilde{3S})$
is generated by adjoining the elements $a_{i}$ and $x_{i}$ to $\mathcal{C}'(I_{4})$
together with the relations $\partial x_{i}=0$. But these imply \[
\rho_{12}=a_{2}=a_{4}\mbox{\ and\ }a_{1}=a_{3}=\rho_{34}\]
together with $\rho_{12}\rho_{34}=a_{1}a_{2}=1$, and so \[
\mathcal{C}'(\widetilde{3S})=\mathbb{F}[x_{1},\dots,x_{5}]\otimes_{\mathbb{F}}(\mathcal{C}'(I_{4})/\langle\rho_{12}\rho_{34}=1\rangle).\]
This is stably isomorphic to the algebra $\mathcal{C}'(\tilde{S})$
computed in the proof of Theorem \ref{thm:S-equals-Z}, so we may
replace three consecutive S tangles with a single one without changing
$\mathcal{C}'(K)$. This move obviously changes the topological type
of $K$, since it removes a full negative twist, but for example we
can conclude that \[
\mathcal{C}'(\mathcal{K}_{-3})\cong\mathcal{C}'(\mathcal{K}_{-5})\cong\mathcal{C}'(\mathcal{K}_{-7})\cong\dots\]
and \[
\mathcal{C}'(\mathcal{K}_{-4})\cong\mathcal{C}'(\mathcal{K}_{-6})\cong\mathcal{C}'(\mathcal{K}_{-8})\cong\dots\]
where $\mathcal{K}_{-n}$ denotes any Legendrian representative of
$K_{-n}$ with maximal $tb$; these are stably isomorphic to $\mathbb{F}[x,y]/\langle(xy+1)^{2}=1\rangle$
and $\mathbb{F}[x,y,z]/\langle(xy+1)z=1\rangle$, respectively.

\begin{acknowledgement*}
I would like to thank my advisor, Tom Mrowka, for his many helpful
questions, suggestions, and conversations about this work. I would
also like to thank Lenny Ng and Josh Sabloff for sharing their knowledge
of the Chekanov-Eliashberg algebra and Legendrian knots in general,
Ana Caraiani for useful discussions about algebraic questions, and the 
referee for many valuable comments.
This work was supported by an NSF Graduate Research Fellowship.
\end{acknowledgement*}

\bibliographystyle{amsplain}

\begin{thebibliography}{99}

\bibitem{Chekanov:2002p539}
Yuri Chekanov, \emph{Differential algebra of {L}egendrian links}, Invent. Math.
  150 (2002), no.~3, 441--483.

\bibitem{Civan:2009p1129}
Gokhan Civan, John~B. Etnyre, Paul Koprowski, Joshua~M. Sabloff, and Alden
  Walker, \emph{Product structures for {L}egendrian contact homology},
  arXiv:math/0901.0490.

\bibitem{Eliashberg:1998p718}
Yakov Eliashberg, \emph{Invariants in contact topology}, Proceedings of the
  {I}nternational {C}ongress of {M}athematicians, {V}ol. {II} ({B}erlin, 1998),
  no. Extra Vol. II, 1998, pp.~327--338 (electronic).

\bibitem{Etnyre:2003p790}
John~B. Etnyre and Lenhard~L. Ng, \emph{Problems in low dimensional contact
  topology}, Topology and geometry of manifolds ({A}thens, {GA}, 2001), Proc.
  Sympos. Pure Math., vol.~71, Amer. Math. Soc., Providence, RI, 2003,
  pp.~337--357.

\bibitem{Etnyre:2002p544}
John~B. Etnyre, Lenhard~L. Ng, and Joshua~M. Sabloff, \emph{Invariants of
  {L}egendrian knots and coherent orientations}, J. Symplectic Geom. 1
  (2002), no.~2, 321--367.

\bibitem{Etnyre:2010p894}
John~B. Etnyre, Lenhard~L. Ng, and Vera V{\'e}rtesi, \emph{Legendrian and
  transverse twist knots}, arXiv:math/1002.2400.

\bibitem{Fuchs:2003p674}
Dmitry Fuchs, \emph{Chekanov-{E}liashberg invariant of {L}egendrian knots:
  existence of augmentations}, J. Geom. Phys. 47 (2003), no.~1,
  43--65.

\bibitem{Fuchs:2004p673}
Dmitry Fuchs and Tigran Ishkhanov, \emph{Invariants of {L}egendrian knots and
  decompositions of front diagrams}, Mosc. Math. J. 4 (2004), no.~3,
  707--717, 783.

\bibitem{Kalman:2006p630}
Tam{\'a}s K{\'a}lm{\'a}n, \emph{Braid-positive {L}egendrian links}, Int. Math.
  Res. Not. 2006, Art. ID 14874, 29 pp.

\bibitem{Lipshitz:2008p649}
Robert Lipshitz, Peter Ozsv{\'a}th, and Dylan Thurston, \emph{Bordered
  {H}eegaard {F}loer homology: {I}nvariance and pairing}, arXiv:math/0810.0687.

\bibitem{Lipshitz:2008p652}
\bysame, \emph{Slicing planar
  grid diagrams: a gentle introduction to bordered {H}eegaard {F}loer
  homology}, Proceedings of {G}\"okova {G}eometry-{T}opology {C}onference 2008,
  G\"okova Geometry/Topology Conference (GGT), G\"okova, 2009, pp.~91--119.

\bibitem{Melvin:2005p636}
Paul Melvin and Sumana Shrestha, \emph{The nonuniqueness of {C}hekanov
  polynomials of {L}egendrian knots}, Geom. Topol. 9 (2005),
  1221--1252 (electronic).

\bibitem{Mishachev:2003p676}
K.~Mishachev, \emph{The {$N$}-copy of a topologically trivial {L}egendrian
  knot}, J. Symplectic Geom. 1 (2003), no.~4, 659--682.

\bibitem{Ng:2001p667}
Lenhard~L. Ng, \emph{Invariants of {L}egendrian links}, Ph.D. thesis,
  Massachusetts Institute of Technology, Cambridge, MA, USA, 2001.

\bibitem{Ng:MR1852765}
\bysame, \emph{Maximal {T}hurston-{B}ennequin number of two-bridge links},
  Algebr. Geom. Topol. 1 (2001), 427--434 (electronic).

\bibitem{Ng:2003p540}
\bysame, \emph{Computable {L}egendrian invariants}, Topology 42
  (2003), no.~1, 55--82.

\bibitem{Ng:MR2186113}
\bysame, \emph{A {L}egendrian {T}hurston-{B}ennequin bound from {K}hovanov
  homology}, Algebr. Geom. Topol. 5 (2005), 1637--1653 (electronic).

\bibitem{Rutherford}
Dan Rutherford, \emph{Thurston-{B}ennequin number, {K}auffman polynomial, and
  ruling invariants of a {L}egendrian link: the {F}uchs conjecture and beyond},
  Int. Math. Res. Not. 2006, Art. ID 78591, 15 pp.

\bibitem{Sabloff:2005p671}
Joshua~M. Sabloff, \emph{Augmentations and rulings of {L}egendrian knots}, Int.
  Math. Res. Not. (2005), no.~19, 1157--1180.

\bibitem{Sabloff:2006p526}
\bysame, \emph{Duality for {L}egendrian contact homology}, Geom. Topol.
  10 (2006), 2351--2381 (electronic).

\bibitem{Stein:sage}
W.\thinspace{}A. Stein et~al., \emph{{S}age {M}athematics {S}oftware
  ({V}ersion 4.3)}, The Sage Development Team, 2009,
  {\tt http://www.sagemath.org}.

\bibitem{Swiatkowski:1992p1071}
Jacek {\'S}wi{\polhk{a}}tkowski, \emph{On the isotopy of {L}egendrian knots},
  Ann. Global Anal. Geom. 10 (1992), no.~3, 195--207.

\end{thebibliography}
\def\polhk#1{\setbox0=\hbox{#1}{\ooalign{\hidewidth
  \lower1.5ex\hbox{`}\hidewidth\crcr\unhbox0}}}
\providecommand{\bysame}{\leavevmode\hbox to3em{\hrulefill}\thinspace}
\providecommand{\MR}{\relax\ifhmode\unskip\space\fi MR }
% \MRhref is called by the amsart/book/proc definition of \MR.
\providecommand{\MRhref}[2]{%
  \href{http://www.ams.org/mathscinet-getitem?mr=#1}{#2}
}
\providecommand{\href}[2]{#2}

\end{document}